\documentclass[12pt]{amsart}
\usepackage{amsthm,amstext,amsmath,amscd,amssymb,latexsym,mathrsfs}
\usepackage{color}
\usepackage{amsfonts,enumerate,array}
\usepackage[colorlinks=true]{hyperref}
\usepackage{todonotes}
\usepackage{verbatim}
\usepackage{float}	
\usepackage{pgf,tikz}
\usepackage[toc,page]{appendix}
\usepackage[all,cmtip]{xy}
\usepackage{stmaryrd}
\usepackage{xparse}
\usepackage{listings}
\usepackage{caption}
\usepackage{subcaption}
\usepackage[T1]{fontenc}
\usepackage{lmodern}
\usepackage[shortalphabetic]{amsrefs}
\usepackage[all]{xy}
\usepackage{mathptmx}
\usepackage{microtype}
\usepackage{framed}
\usepackage{graphicx}
\usepackage{wasysym}

\usepackage[centering, includeheadfoot, hmargin=1in, vmargin=1.5in, headheight=30.4pt]{geometry}

\newcommand{\CC}{{\mathbb C}}
\newcommand{\R}{{\mathbb R}}

\newcommand{\PP}{{\mathbb P}}

\newcommand{\p}{{\mathfrak p}}

\newcommand{\Z}{{\mathbb Z}}
\newcommand{\N}{{\mathbb N}}

\newcommand{\QQ}{{\mathbb Q}}
\newcommand{\OO}{{\mathcal O}}
\newcommand{\mC}{{\mathcal C}}
\newcommand{\mD}{{\mathcal D}}
\newcommand{\mR}{{\mathcal R}}

\newcommand{\cL}{{\mathcal L}}
\newcommand{\cB}{{\mathcal B}}

\NewDocumentCommand{\dslash}{s}{%
  \IfBooleanTF{#1}
    {\big/\mkern-7mu\big/}
    {/\mkern-6mu/}%
}

\DeclareMathOperator{\Spec}{Spec}
\DeclareMathOperator{\Proj}{Proj}
\DeclareMathOperator{\Div}{Div}
\DeclareMathOperator{\HHom}{Hom}

\DeclareMathOperator{\im}{im}

\DeclareMathOperator{\Pic}{Pic}
\DeclareMathOperator{\Cl}{Cl}

\DeclareMathOperator{\Eff}{Eff}
\DeclareMathOperator{\Nef}{Nef}

\DeclareMathOperator{\Mov}{Mov}

\DeclareMathOperator{\Cox}{Cox}

\DeclareMathOperator{\Rel}{Rel}

\DeclareMathOperator{\rk}{rk}

\newcommand{\paper}{: \begin{it}}
\newcommand{\jour }{, \end{it}}

\newtheorem{theorem}{Theorem}[section]
\newtheorem{thm}[theorem]{Theorem}
\newtheorem{lemma}[theorem]{Lemma}
\newtheorem{proposition}[theorem]{Proposition}
\newtheorem{corollary}[theorem]{Corollary}

\theoremstyle{definition}
\newtheorem{definition}[theorem]{Definition}

\newtheorem{example}[theorem]{Example}

\newtheorem*{theoremA}{Theorem A}

\theoremstyle{remark}
\newtheorem{remark}[theorem]{Remark}
\newtheorem{rmk}[theorem]{Remark}

\numberwithin{equation}{section}

\newcommand{\equ}{\ensuremath{\,=\,}}

\newcommand{\dleq}{\ensuremath{\,\leq\,}}
\newcommand{\deq}{\ensuremath{\stackrel{\textrm{def}}{=}}}
\newcommand{\sC}{\ensuremath{\mathscr{C}}}

\newcommand{\vol}[2]{\ensuremath{{\rm vol}_{#1}( #2 ) } }
\newcommand{\dsubseteq}{\ensuremath{\,\subseteq\,}}

\DeclareMathOperator{\Supp}{Supp}

\DeclareMathOperator{\Quot}{Quot}

\newcommand{\cA}{\ensuremath{\mathcal A}}
\newcommand{\cO}{\ensuremath{\mathcal O}}
\DeclareMathOperator{\intt}{int}
\DeclareMathOperator{\Effb}{\overline{\mathrm{Eff}}}
\newcommand{\rat}{\ensuremath{\dashrightarrow}}
\newcommand{\mcal}{\mathcal}

\newcommand{\lto}{\longrightarrow}
\DeclareMathOperator{\Mob}{Mob}
\DeclareMathOperator{\sB}{\mathbf{B}}

\newcommand{\lra}{\ensuremath{\longrightarrow}}

\title[Geometry of multigraded rings]{Geometry of multigraded rings and embeddings into toric varieties}

\author{Alex K\"uronya}
\address{Institut f\"ur Mathematik, Goethe-Universit\"at Frankfurt, Robert-Mayer-Str. 6-10, D-60325 Frankfurt am Main, Germany }
\email{kuronya@math.uni-frankfurt.de}
\author{Stefano Urbinati}
\address{Dipartimento di Scienze Matematiche, Informatiche e Fisiche , Universit\`a degli Studi di Udine, Via delle Scienze 206, Udine, Italy}
\email{urbinati.st@gmail.com}

\dedicatory{Dedicated to Bill Fulton on the occasion of his 80\textsuperscript{th} birthday.}


\date{}


\begin{document}
\begin{abstract}
We use homogeneous spectra of multigraded rings to construct toric embeddings of a large family of projective varieties  which preserve some of the birational geometry of the underlying variety, generalizing the well-known construction associated to Mori Dream Spaces. 
\end{abstract}

\maketitle

\tableofcontents

\section*{Introduction}

 In this paper our principal aim is to construct embeddings of projective varieties into simplicial toric ones which partially preserve the birational geometry of the underlying spaces. 
This we plan to achieve by  pursuing two  interconnected trains of thought: on one hand we study    geometric spaces that arise from multigraded rings, and on the other hand we  analyze the relationship between local Cox rings  and embeddings of projective varieties into toric ones. The latter contributes to our understanding of birational maps between varieties along the lines of \cites{HK,KKL,LV}.

It has been known for a long time  how to build  a projective variety or scheme starting from a ring graded by the natural numbers. This  construction gave rise to the  very satisfactory theory of projective spectra, which has become absolutely fundamental to modern algebraic geometry. Brenner and Schr\"oer in \cite{BS} pointed out how to generalize the above process to rings that are graded by finitely generated abelian groups. The class of schemes that arise in their work are still reasonably well-behaved, but they tend to be non-separated, which appears to be an obstacle in applications. Here we show how to utilize the theory in \cite{BS} to construct toric embeddings for projective varieties. 

While the homogeneous spectrum of an $\N$-graded ring gives an interesting interplay between commutative algebra and algebraic geometry already, in the multigraded case a new aspect arises: convex geometry starts playing a significant role. More concretely,  one associates convex cones in the vector space spanned by the grading group to elements of the underlying ring, and geometric properties --- most importantly, separability --- will depend on the relative positions of these cones.

The philosophical starting point of our work is the article  \cite{HK} of Hu--Keel on Cox rings of Mori dream spaces. As shown in \cite{HK}, Mori dreams spaces are precisely the varieties
whose Cox rings (or total coordinate rings) are finitely generated. In proving their results, Hu and Keel rely on an intricate analysis of the variation of GIT quotients (cf. \cites{DH,Th}), however the equivalence of being a Mori dream space and having a finitely generated Cox rings can be demonstrated my methods of the minimal model program as well \cite{KKL}.
The cone of (pseudo)effective divisors of a Mori dream space admits a finite decomposition into finite rational polyhedral chambers, which correspond to birational models of the variety. 
Cox rings connect algebraic geometry to combinatorics and even number theory in the form of universal torsors \cites{CoxRingsBook,CT,HTsch}. 

To be more concrete, by definition,  a Mori dream space is a normal projective $\QQ$-factorial variety satisfying the following conditions
\begin{enumerate}
	\item  $N^1(X)_\QQ\simeq \Pic(X)_\QQ$,
	\item  the cone of movable divisors is a cone over a rational polytope, and 
	\item  the nef cone is generated by finitely many semi-ample divisors.
\end{enumerate}

The issue of finite generation put aside, the first of the three conditions is surprisingly restrictive, in particular it implies that divisors numerically equivalent to finitely generated ones are themselves finitely generated. It was realized in \cite{KKL} that this latter condition is very important for certain variants of the minimal model program to work, and at the same time it does not hold in general (cf. Example~\ref{eg:fin gen not invariant num}).  

A divisor $D$ as above, that is, having the property that the section ring of  every divisor $D'$ numerically equivalent to $D$ is finitely generated, are called \emph{gen} in \cite{KKL} (see Definition~\ref{defn:gen}). Ample and adjoint divisors on projective varieties are gen, this is how gen divisors typically occur in nature when the condition $N^1(X)_\QQ\simeq \Pic(X)_\QQ$ is not  met.  

An important contribution of the article \cite{HK} in their Proposition 2.11  is the construction of an embedding of an arbitrary Mori dream space $X$ into a simplicial (normal projective) toric variety $Z$ having the property that the Mori chamber structure and the birational geometry of $X$ is realized via pulling back chambers or birational maps from $Z$. This underlines the observation of Reed that flips (at least under strong finite generation hypotheses) should come from a toric situation. The Hu--Keel embedding was used in \cite{PU} in connection with construction of maximal rank valuations and indentifying Chow quotients of Mori dream spaces (cf. \cite{KSZ}). 

One way to generalize the results of Hu and Keel are to consider cones of divisors with finitely generated local Cox rings. Along with the issue of gen divisors (which prevents the methods of \cite{HK} to be directly applicable in this more general context), this point of view was taken up in \cite{KKL}. It was demonstrated in  \emph{loc. cit.} that cones with finitely generated Cox rings are the natural habitat of the minimal model program. 

Following the line of thought of \cite{KKL}, we pursue establishing a local version of the embedding result of \cite{HK}. We will study local Cox rings (certain finitely generated subrings of total coordinate rings) on normal projective varieties and construct toric embeddings which recover the Mori chamber structure and birational maps of the part of the cone of effective divisor covered by the local Cox ring in question. 

Observe that chamber structures analogous to the one on Mori dream spaces exist in the abscence of such strong finite generation hypotheses as well. On surfaces the variation of Zariski decomposition yields a locally finite chamber decomposition without any sort of finite generation condition \cite{BKS}, and it has been known for several decades that the parts of the pseudo-effective cone of a projective variety that is seen by the Mori program decomposes into finite rational polyhedral chambers. 

It is important to point out that we do not require the Picard group to be finitely generated, hence the methods of Hu--Keel do not apply directly. Although  our arguments are in a sense modeled on those of \cite{HK}, as we will see, our  strategy and our  tools are  different. 

Our main result is the following. 

\begin{theoremA}[Theorem~\ref{thm:embedded MMP}]	
Let $X$ be a $\QQ$-factorial normal projective variety over the complex numbers, $\mC\subseteq\Div(X)$ a finitely generated cone containing an ample divisor such that   its image in $N^1(X)_\R$ has maximal dimension and every divisor in the cone is gen. Then there exists a closed embedding $X \subseteq W$ into a simplicial projective toric variety such that
\begin{enumerate}
		\item the restriction $N^1(W)_\R \to N^1(X)_\R$ is an isomorphism;
		\item the isomorphism of (1) induces an isomorphism $NE^1(W)\to \mC \subseteq N^1(X)_\R$.
 
		\item If in addition every divisor in $\mC$ is gen then every Mori chamber of $X$ is a union of finitely many Mori chambers of W, and 
		\item for every rational contraction $f: X\dashrightarrow X'$ there is toric rational contraction $\overline f: W \dashrightarrow W'$, regular at the generic point of $X$, such that $f=\overline{f}|_X$.
\end{enumerate}
\end{theoremA}

The idea of the construction relies on homogeneous spectra of multigraded rings. The concrete setup goes as follows: let $D$ be a finitely generated abelian group, $S$ a $D$-graded ring. Brenner--Schr\"oer then define the homogeneous spectrum $\Proj^DS$ as a certain subspace of a quotient of $\Spec S$ in the sense of locally ringed spaces.  Much like  a surjection of $\N$-graded rings $k[T_0,\dots,T_r] \twoheadrightarrow A$ (where $k$ is an arbitrary field) gives rise to a closed  embedding $\Proj A \hookrightarrow \PP^{r}$, given a $D$- graded ring $S$, a surjective homomorphism  
\[
k[T_0,\dots,T_r] \longrightarrow S\ ,
\]
where we consider $k[T_0,\dots,T_r]$ is given the induced grading, yields an embedding 
\[
\Proj^D S \ \hookrightarrow\ \Proj^D k[T_0,\dots,T_r]\ ,
\]
where the latter is a simplicial toric prevariety. 

A significant difference between the $\N$-graded and the multigraded case is that in the latter $\Proj^DS$ is often not going to be separated. For a concrete example (see \ref{ex:Proj Cox nonsep}), let $X$ be the blow-up of $\PP^2$ at a point, $D=\Pic(X)$, and $S=\Cox(X)$. Then $\Proj^DS$ is not separated, in particular it is \emph{not} isomorphic to $X$. The non-separability  phenomenon proves to be a major hurdle, and a large part of our work goes into dealing with it. \\

\paragraph{\bf Organization of the article.}

In Section 1 we quickly recall (mainly following \cite{KKL}) the basic results and terminology around divisorial rings and local Cox rings. Section 2 is devoted to the basics of multigraded rings, relevant elements, and the first connections with convex geometry. In Section 3 we present the construction of multihomogeneous spectra from \cite{BS} and expand it to fit our needs. Section 4 hosts a discussion of ample families and constructions of closed embeddings of spectra into (separated) varieties. The last two sections deal with gen divisors and the proof of Theorem~A. \\

\paragraph{\bf Notation.}

Throughout the paper $D$ denotes a finitely generated abelian group, and $S$ denotes a $D$-graded ring (always commutative with $1$). The notation $\Rel^D(S)$ stands for the set of relative elements with respect to the given $D$-grading (see Definition~\ref{defn:relevant}), while $\Proj^D(S)$ denotes the multihomogeneous spectrum of the $D$-graded ring $S$ (see Definition~\ref{defn:multihomogeneous spectrum}). The map $\delta\colon D\to D'$ is a group homomorphism (used for regrading $D$-graded rings). \\

\paragraph{\bf Acknowledgements.}

We would like to thank Dave Anderson,  Andreas B\"auerle, Michel Brion, Mihai Fulger, Vlad Lazi\'c, Alex Massarenti, Yusuf Mustopa, Joaquim Ro\'e, Jakob Stix, Martin Ulirsch,  and Torsten Wedhorn   for helpful discussions. The first author was partially supported by the  LOEWE grant ‘Uniformized Structures in Algebra and Geometry'.

\section{Divisorial rings and Cox rings}

In this preliminary section we collect some  definitions and facts about (multigraded) divisorial rings that we will need in the later sections (mostly Sections 5 and 6) of  the paper. Our main sources are \cites{HK,KKL,LV}, at the same time we point the reader to \cite{CoxRingsBook} for a massive amount of information along these lines. The main purpose of the section is to fix notation and terminology. 

Let $X$ be a normal projective variety over an algebraically closed field, and let  $\mathcal S\subseteq\Div_{\mathbb Q}(X)$ is a finitely generated monoid. 
We say that 
\[
R(X,\mathcal{S}) \deq \bigoplus_{\mD\in\mathcal{S}}H^0\big(X,\mD)
\]
is the {\em divisorial ring associated to $\mathcal S$}. If $\mC\subseteq\Div_{\mathbb R}(X)$ is a rational polyhedral cone then $\mathcal S \deq \mC\cap\Div(X)$ is a finitely generated monoid by Gordan's lemma. We define  $R(X,\mC)$ to be $R(X,\mathcal S)$. We also work with  divisorial rings of the form
\[
\mR \equ R(X;\mD_1, \dots, \mD_r) \equ \bigoplus_{(n_1,\dots, n_r)\in \N^r} H^0(X,
n_1\mD_1+\dots + n_r\mD_r)\ ,
\]
where $\mD_1,\dots, \mD_r\in\Div_\QQ(X)$. If $\mD_i$ are adjoint divisors then  the ring $\mR$ is called an {\em adjoint ring\/}. The {\em support\/} of $\mR$ is the cone
\[
\Supp \mR \deq \{\mD\in\sum\R_+\mD_i\mid |\mD|_\R\neq\emptyset\}\subseteq \Div_\R(X)\ ,
\]
and similarly for rings of the form $R(X,\mC)$.

If $X$ is a $\QQ$-factorial projective  variety with $\Pic(X)_\QQ= N^1(X)_\QQ$, and if $\mD_1, \dots, \mD_r$ is a basis of $\Pic(X)_\QQ$ such that $\Effb(X)\subseteq \sum \R_+\mD_i$, then $R(X;\mD_1,\dots,\mD_r)$ is a \emph{Cox ring} of $X$. The isomorphism class (in particular, whether it is finitely generated) of this ring is independent of the choice of $\mD_1,\dots,\mD_r$, hence  we will somewhat loosely call it \emph{the Cox ring} or \emph{the total coordinate ring  of $X$} and denote it by $\Cox(X)$ (cf. \cites{HK,LV}). 

\begin{definition}[Mori dream space]
A projective $\QQ$-factorial variety $X$ is a Mori dream space if
\begin{enumerate}
	\item $\Pic(X)_\QQ =N^1(X)_\QQ$,
	\item $\Nef(X)$ is the affine hull of finitely many semiample line bundles, and
	\item there are finitely many birational maps $f_i\colon X \dashrightarrow X_i$ to projective $\QQ$-factorial varieties $X_i$ such that each $f_i$ is an isomorphism in codimension $1$, each $X_i$ satisfies (2), and $\overline{\Mov}(X)=\bigcup f^*_i\big(\Nef(X_i)\big)$.
\end{enumerate}
\end{definition}

\begin{thm}[\cite{HK} Theorem 2.9]
Let $X$ be a normal $\QQ$-factorial projective variety. Then $X$ is a Mori dream space if and only if $\Cox(X)$ is a finitely generated ring. 	
\end{thm}

While studying  multigraded rings, we will consider gradings by finitely generated abelian groups. In the examples above, we will work with the gradings given by 
the group generated by ${\mathcal S}$, or, in most cases, the grading by the N\'eron--Severi group $N^1(X)$.

\begin{rmk}
In general the Cox ring of $X$ is not going to be finitely generated even if $\Pic(X)\simeq N^1(X)$. In addition, when viewed as an $N^1(X)$-graded ring, the individual graded pieces of $\Cox(X)$ will typically not be of finite dimension over the base field either. In fact, if $d\in N^1(X)$, then 
\[
\Cox(X)_d \equ \bigoplus_{\cL\in d} H^0(X,\cL)\ ,
\]
where we sum over all line bundles in the numerical equivalence class $d$. Recall \cite{PAGI}*{Proposition 1.4.37} that these are parametrized by a scheme of finite type, but not one with finitely many points in general. 
\end{rmk}

\begin{definition}[Local Cox ring]
	Let $X$ be a normal projective $\QQ$-factorial variety, and let $\mC\subseteq \Div_\R(X)$
	rational polyhedral cone such that the ring $R\deq R(X,\mathcal C)$ is finitely generated. Then we call $R$ a \emph{local Cox ring of $X$}. 	
\end{definition}

Throughout the paper, we use several properties of finitely generated divisorial rings without explicit mention, see \cite{CaL10}*{\S 2.4} for details and background.  

\begin{lemma}\label{lem:3}
	Let $X$ be a normal projective variety, let $\mD_1,\dots,\mD_r$ be divisors in $\Div_\QQ(X)$, and let $p_1,\dots,p_r$ be positive rational numbers. 
	
	Then $R(X;\mD_1,\dots,\mD_r)$ is finitely generated if and only if $R(X;p_1\mD_1,\dots,p_r\mD_r)$ is finitely generated. 
\end{lemma}

Finite generation of a divisorial ring $\mR$ has important consequences on the behavior of the asymptotic order functions, and therefore on the convex geometry of $\Supp \mR$, as observed in \cite{ELMNP} (cf. \cite{KKL}) .   

\begin{thm}
	\label{thm:ELMNP}
	Let $X$ be a projective $\QQ$-factorial variety, and let $\mathcal C\subseteq\Div_\R(X)$ be a rational polyhedral cone. Assume that the ring $\mathfrak R=R(X,\mcal C)$ is finitely generated. Then:
	\begin{enumerate}
		\item $\Supp \mathfrak R$ is a rational polyhedral cone,
		\item if $\Supp \mathfrak R$ contains a big divisor, then all pseudo-effective divisors in $\Supp\mathfrak R$ are in fact effective,
		\item there is a finite rational polyhedral subdivision $\Supp \mathfrak R=\bigcup \mcal{C}_i$ such that $o_\Gamma$ is linear on $\mcal{C}_i$ for every geometric valuation $\Gamma$ over $X$, and the cones $\mcal C_i$ form a fan,
		\item there is a positive integer $d$ and a resolution $f \colon \widetilde{X} \lto X$ such that $\Mob f^*(d\mD)$ is basepoint free for every $\mD\in \Supp \mathfrak R\cap \Div(X)$, and $\Mob f^*(kd\mD)=k\Mob f^*(d\mD)$ for every positive integer $k$.
	\end{enumerate}
\end{thm}
\begin{proof} 
For a proof see \cite{ELMNP}*{Theorem 4.1} or  \cite{CoL10}*{Theorem 3.6}.
\end{proof}

Part (i) of the following lemma is \cite{CoL10}*{Lemma 3.8}. Part (ii) is a result of Zariski and Wilson, cf.\ \cite{PAGI}*{Theorem 2.3.15}.

\begin{lemma}\label{lem:ords}
	Let $X$ be a normal projective variety and let $\mD$ be a divisor in $\Div_\QQ(X)$.
	\begin{enumerate}
		\item[(i)] If $|\mD|_\QQ\neq\emptyset$, then $\mD$ is semiample if and only if $R(X, \mD)$ is finitely generated and $o_{\Gamma}(\mD)=0$ for all geometric valuations $\Gamma$ over $X$.
		\item[(ii)] If $\mD$ is nef and big, then $\mD$ is semiample if and only if $R(X, \mD)$ is finitely generated.
	\end{enumerate}
\end{lemma}
\begin{proof} 
	If $\mD$ is semiample, then some multiple of $D$ is basepoint free, thus $R(X,\mD)$ is finitely generated by Lemma \ref{lem:3}, and all $o_\Gamma(\mD)=0$. Now, fix a point $x\in X$. If $R(X,\mD)$ is finitely generated and $o_x(\mD)=0$, then $x\notin\sB(\mD)$ by Theorem \ref{thm:ELMNP}(4), which proves (i).
	
	For (ii), let $\cA$ be an ample divisor. Then $\mD+\varepsilon \cA$ is ample for any $\varepsilon>0$, hence $o_\Gamma(\mD+\varepsilon \cA)=0$ for any geometric valuation $\Gamma$ over $X$. But then $o_\Gamma(\mD)=\lim\limits_{\varepsilon\to0}o_\Gamma(\mD+\varepsilon \cA)=0$,  hence we are done by part (i).
\end{proof}

The following statement taken from \cite{KKL}  will not be used directly in our paper, nevertheless, since it is the source of all other polyhedral decomposition results (for instance \cite{KKL}*{Theorem 4.2} or \cite{KKL}*{Theorems 5.2 and 5.4}), we include it for the sake of completeness.  

\begin{corollary}\label{cor:5}
	Let $X$ be a normal projective variety and let $\mD_1, \dots, \mD_r$ be divisors in $\Div_\QQ(X)$. Assume that the ring $\mathfrak R=R(X;\mD_1, \dots, \mD_r)$ is finitely generated, and let $\Supp \mathfrak R=\bigcup_{i=1}^N \mcal{C}_i$ be a finite rational polyhedral subdivision as in Theorem~\ref{thm:ELMNP}(3). Denote by $\pi\colon\Div_\R(X)\lto N^1(X)_\R$ the natural projection.
	
	Then there is a set $I_1\subseteq\{1,\dots,N\}$ such that $\Supp\mathfrak R\cap \pi^{-1}\bigl(\overline{\Mov} (X)\bigr)=\bigcup_{i\in I_1}\mathcal C_i$.
	
	Assume further that $\Supp\mathfrak R$ contains an ample divisor. Then there is a set $I_2\subseteq\{1,\dots,N\}$ such that the cone $\Supp\mathfrak R\cap \pi^{-1}\bigl(\Nef(X)\bigr)$ equals $\bigcup_{i\in I_2}\mathcal C_i$, and every element of this cone is semiample.
\end{corollary}
\begin{proof}
	For every prime divisor $\Gamma$ on $X$ denote $\mathcal C_\Gamma=\{\mD\in\Supp\mathfrak R\mid o_\Gamma(\mD)=0\}$. If $\mC_\Gamma$ intersects the interior of some $\mC_\ell$, then $\mcal C_\ell\subseteq\mcal C_\Gamma$ since $o_\Gamma$ is a linear non-negative function on $\mcal C_\ell$. Therefore, there exists a set $I_\Gamma \subseteq\{1,\dots,N\}$ such that $\mathcal C_\Gamma=\bigcup_{i\in I_\Gamma}\mathcal C_i$. Now the first claim follows since $\overline{\Mov}(X)$ is the intersection of all $\mathcal C_\Gamma$.
	
	For the second claim, note that since $\Supp\mathfrak R\cap \pi^{-1}\bigl(\Nef(X)\bigr)$ is a cone of dimension $\dim\Supp\mathfrak R$, we can consider only maximal dimensional cones $\mcal C_\ell$. Now, for every $\mcal C_\ell$ whose interior contains an ample divisor, all asymptotic order functions $o_\Gamma$ are identically zero on $\mcal{C}_\ell$ similarly as above. Therefore, by Lemma \ref{lem:ords}, every element of $\mcal{C}_\ell$ is semiample, and thus $\mcal{C}_\ell\subseteq \Supp\mathfrak{R}\cap\pi^{-1} \bigl(\Nef(X)\bigr)$. The claim follows.
\end{proof}

\section{Relevant elements in multigraded rings}
 
This section is devoted to the basic algebraic theory of multigraded rings. As far as the fundamental definitions go, we follow \cite{BS}*{Section 2}. The concrete  examples we have in mind are polynomial rings over a noetherian ring, and divisorial algebras over algebraically closed fields of characteristic zero.  
 
The notion of a relevant element (in Definition~\ref{defn:relevant}) was coined by Brenner--Schr\"oer. Relevant elements in a multigraded ring are precisely the ones which lead to localizations with good properties, that is, where a geometric quotient will exist. As such, they are instrumental in the construction of multihomogeneous spectra of rings (cf. Section~\ref{sec:multigraded proj}). The notion  of a relevant element brings in convex geometry  in a crucial way (cf. Remark~\ref{rem:relevant and cone}). 

Let $D$ be a finitely generated abelian group, and let
\[
S \equ \bigoplus_{d\in D} S_d
\]
be a $D$-graded ring. An element $s\in S$ is called homogeneous if it belongs to one of the $S_d$'s, in that case the degree of $s$ is defined to be $d$. A morphism of $D$-graded rings 
is a degree-preserving ring  homomorphism. 

As explained in \cite{SGA_I}*{I.4.7.3}  the $D$-grading on $S$ corresponds to an action of the diagonalizable group scheme $\Spec (S_0[D])$ on the affine scheme $\Spec (S)$.
We will be especially interested in the case $D=\Z^n$, which gives rise to torus actions. 

\begin{definition}
Let $\delta\colon D\to D'$ be a homomorphism between finitely generated abelian groups, $S$ a ring equipped with a $D$ and a $D'$ grading. We say that the two gradings are \emph{compatible}
if  $\delta(\deg_{D}(f)) = \deg_{D'}(f)$ holds  for every homogeneous element $f\in S$.  
\end{definition}

\begin{remark}\label{rmk:change of rings}
Let $D$, $D'$ be finitely generated abelian groups,  $S$ be a $D$-graded ring, and let $\delta\colon D\to D'$ be a group homomorphism. Then $\delta$ equips $S$ with the structure of a $D'$-graded ring via
\[
S \equ \bigoplus_{d'\in D'} \left( \bigoplus_{d\in\delta^{-1}(d')} S_d\right)\ .
\]
We write $S_\delta$ for $S$ when thought of as a graded ring via $D'$ (or $\delta$, to be more precise). Note that the grading induced by $\delta$ is compatible with  the $D$-grading. 
Regradings by a group homomorphism $\delta$ will be often  of interest when $\delta$ is surjective. 
\end{remark}

\begin{example}
Of particular importance for us are regradings of the type $\delta\colon \Z^k\to\Z^r$. A concrete example  is the case when $X$ is a projective variety, $\mD_1,\dots,\mD_r$ are divisors on $X$ and we consider the homomorphism
\begin{eqnarray*}
\Z^r \deq  \Z \mD_1\oplus\ldots\oplus \Z \mD_r & \longrightarrow & N^1(X) \\
 m_1\mD_1+\ldots+m_r\mD_r & \mapsto & \text{numerical equivalence class of 	$ m_1\mD_1+\ldots+m_r\mD_r$}\ .
\end{eqnarray*}	
\end{example}

\begin{definition}\label{defn:relevant}
We call a $D$-graded ring $S$ \emph{periodic} if the subgroup
\[
\{ \deg(f) \mid f\in S^\times \text{ homogeneous } \} \dleq D
\]
is of finite index. An element of a $D$-graded ring $S$ is said to be \emph{relevant} if it is homogeneous, and the localization $S_f$ is periodic. We denote the set of relevant elements of $S$ by $\Rel^D(S)$. 	
\end{definition}

\begin{remark}
The significance of periodic $D$-graded rings stems from their property that the morphism $\Spec S \to \Spec S_0$ induced by the inclusion $S_0\hookrightarrow S$ is a geometric quotient (see \cite{BS}*{Lemma 2.1}). In particular, for $f\in\Rel^D(S)$, the morphism $S_f\to S_{(f)} \deq (S_f)_0$ will be a geometric quotient.  If $D\simeq \Z$, then every element is periodic. 
\end{remark}

\begin{lemma}
Let $D$ be a finitely generated abelian group, $f\colon S\to R$ a  surjective morphism of $D$-graded rings. Then
\begin{enumerate}
	\item if $S$ is periodic then so is $R$;
	\item if $t\in S$ is relevant, then $f(t)\in R$ is relevant as well.   
\end{enumerate}
\end{lemma} 
\begin{proof}
For the first statement one readily checks the following: let   $t\in S$ a nonzerodivisor, then  $R_{f(t)}$ is periodic provided    $S_t$ was  periodic to begin with.  	
The second statement follows immediately from the first one. 	  
\end{proof}

\begin{definition}[Cone of a relevant element]
Let $D$ be a finitely generated abelian group, $S$ a $D$-graded ring. Given a homogeneous element $f\in S$, the \emph{cone of $f$} is defined as 
 \begin{eqnarray*}
 	\sC(f) & \deq & \text{closed convex cone in $D\otimes_\Z\R$ generated by } \\
 	&& \{\deg(g)\mid \text{$g$ is a homogeneous divisor of $f^m$ for some $m\in\N$} \}\ .
 \end{eqnarray*}	
\end{definition}

\begin{remark}\label{rem:relevant and cone}
A homogeneous element $f\in S$ is relevant precisely if the degrees of its homogeneous divisors generate a finite index subgroup of $D$, which is in turn equivalent to  $\intt \sC(f)\subseteq D_\R \deq D\otimes_\Z\R$ being non-empty. It follows for instance that the product of relevant elements remains relevant. 
\end{remark} 

\begin{remark}\label{rem:subrings}
Let $S$ be a $D$-graded ring, $D_1\leq D$ a subgroup. Then we can define the  subring
\[
S^{D_1} \deq \deg_{D}^{-1}(D_1) \equ \bigoplus_{d\in D_1} S_d\ .
\]
The ring $S^{D_1}$ obtains a grading by $D_1$,  an element $f\in S^{D_1}$ is $D$-homogeneous if and only if it is $D_1$-homogeneous.

\end{remark}

\begin{lemma}\label{lem:degree in the interior}
Let $D$ be a finitely generated abelian group, $S$ a $D$-graded  ring, $f\in S$ a homogeneous element. If $f\in S$ is relevant, then $\deg (f)\in \intt \sC(f)$. 
\end{lemma}
 \begin{proof}
 	Let $g_1,\dots,g_r$ be a finite collection of homogeneous divisors of powers of  $f$ such that 
 	the interio of the convex span of $\{\deg(g_i)\mid 1\leq i\leq r\}$ is contained in the interior of $\sC(f)$. For any index $1\leq i\leq r$ write  
 	\[
 	f^{m_i} \equ g_ih_i\ ,
 	\]
 	where $h_i$ is a homogeneous divisor of $f^{m_i}$ as well. Then 
 	\[
 	f^{m_1+\dots+m_r} \equ (g_1 \cdot\ldots\cdot g_r)\cdot (h_1\cdot\ldots\cdot h_r)\ ,
 	\]
 	consequently,
 	\[
 	\deg(f^{m_1+\dots+m_r}) \equ \sum_{i=1}^{r} \deg(g_i) + \sum_{i=1}^{r}\deg(h_i)\ . 
 	\]
 	Of the two sums on the right-hand side, the first one lies in the interior of $\sC(f)$ by construction, while the second one belongs to $\sC(f)$ as well. Therefore, the left-hand side lies in the interior of $\sC(f)$ as well. 
 \end{proof}
 
\begin{remark}
An important consequence of Lemma~\ref{lem:degree in the interior} is that even if $S$ is a finitely generated graded ring over a field, $S$ might not possess a finite system of generators consisting of relevant elements. This is caused by the fact that if $\rk\, D\geq 2$, degrees of relevant elements do not lie on the boundary of the cone spanned by the degrees of $S$.  
\end{remark} 
 
\begin{remark}\label{rem:product is relevant}
Let $S$ be a $D$-graded ring, let  $f_1,\dots,f_r\in S$ be homogeneous elements which form a $\QQ$-basis of $D$, and let  $f = f_1\cdot\ldots\cdot f_r$. Then $\deg(f)=\deg(f_1)+\ldots+\deg(f_r)$ is contained in the interior of the closed convex cone in $D_\R$ spanned by the basis $\deg(f_1),\ldots,\deg(f_r)$, hence $\deg(f)\in \intt \sC(f)$, that is, $f$ is relevant. 	
\end{remark}

\begin{lemma}\label{lem:relevance after regrading}
Let $\delta\colon D\to D'$ be a surjective group homomorphism between finitely generated abelian groups, $S$ a $D$-graded ring. Write $\delta_\R$ for the induced surjective linear map $D\otimes_\Z \R\to D'\otimes_\Z \R$. Let $f\in S$.  
\begin{enumerate}
	\item If $f$ is $D$-homogeneous then it is $\delta$-homogeneous as well. The converse fails whenever $\delta$ has non-trivial kernel. 
	\item If $f$ is $D$-homogeneous, then we have  the equality  $\delta_\R(\sC(f)) \equ  \sC_\delta(f)$.
	\item If $f$ is  $D$-relevant then it is  $\delta$-relevant as well.  
\end{enumerate}
\end{lemma}
\begin{proof}
The first statement follows from the definition of the $\delta$-grading of $S$. The second one follows from $(1)$,  the additivity of the grading, and the fact that taking closed convex hulls commutes with surjective linear maps. Last, $(3)$ follows from $(2)$ and  the fact that surjective linear maps are open.
\end{proof}

\begin{example}
Consider $S=k[T_1,T_2]$ with the gradings by $D=\Z^2$ according to monomial degree, and let $\delta\colon \Z^2\to\Z$ be the total degree map. However, the element $T_1^2$ 
is $D'$-relevant (as every monomial in $k[T_1,T_2]$ is), but $T_1^2$ is a homogeneous but \emph{not} relevant element in $S$ graded by $D$.   
\end{example}

\section{Homogeneous spectra of multigraded rings}\label{sec:multigraded proj}

An important  ingredient of our work is the construction of homogeneous spectra of multigraded rings itroduced by  Brenner--Schr\"oer \cite{BS}. We recall the main features, partially to fix  notation and to obtain the precise statements we need. We roughly follow the exposition in \cite{BS}, where we   refer the reader for further detail.  Although this section has an  expository flavour, much of the material is in fact new. 

\begin{lemma}[\cite{BS} Lemma 2.1]
	Let $S$ be a periodic $D$-graded ring. Then the projection morphism $\Spec (S)\to \Spec (S_0)$ is a geometric quotient in the sense of geometric invariant theory. 
\end{lemma}

Homogeneous localization by relevant elements leads to open subschemes
\[
D_+(f) \deq \Spec (S_{(f)}) \dsubseteq \Quot(S)\ ,
\]
which in turn yields the following. 

\begin{definition}[Multihomogeneous spectrum of a multigraded ring]\label{defn:multihomogeneous spectrum}
Let $D$ be a finitely generated abelian group, and let $S$ be a $D$-graded ring. We define the \emph{(multi)homogeneous spectrum of $S$} as the scheme 
\[
\Proj^D S \deq \bigcup_{f\in S \text{ relevant}} D_+(f) \dsubseteq \Quot(S)\ .
\]
\end{definition}

\begin{remark}
One can verify that the above  definition coincides with the more traditional one obtained by gluing the collection of affine schemes
\[
\{ \Spec S_{(f)} \mid f\in \Rel^D(S) \}\ ,
\]
via the inclusions $\Spec S_{(fg)} \hookrightarrow \Spec S_{(f)}$ for all pairs of $D$-relevant elements $f,g\in S$ (recall from Remark~\ref{rem:relevant and cone} that the product of relevant elements is again relevant). 

In the case $D=\Z$ and $\deg_D(S)\subseteq\N$ we recover the usual definition of $\Proj S$. Although this gives rise to a slight discrepancy in our notation, we will write $\Proj^\N S$ to denote the classical projective spectrum. The reason for distinguishing between $\N$ and $\Z$-gradings is that the latter may yield a non-separated result. 
\end{remark}

\begin{remark}
As expected, one defines the irrelevant ideal of $S$ via
\[
S_+ \deq ( f\in S\mid \text{$f$ is relevant}) \lhd S\ ,
\]
and calls $V(S_+)\subseteq \Proj (S)$ the irrelevant subscheme. One obtains an affine projection morphism
\[
\Spec (S) \setminus V(S_+) \longrightarrow \Proj(S)\ ,
\]
which is a geometric quotient for the induced action. 

The points of $\Proj (S)$ correspond to graded (not necessarily prime!) ideals $\p \lhd S$ not containing $S_+$ and such that the subset of homogeneous elements 
$H\subseteq S\setminus \p$ is closed under multiplication. 
\end{remark}

\begin{lemma}
	Let $S$ be a $D$-graded ring, $m$ a positive integer. Write 
	\[
	(S^{(m)})_d \deq S_{md}\ .
	\]
	Then $\Proj^DS \simeq \Proj^D S^{(m)}$. 
\end{lemma}

\begin{lemma} \label{lemma:dimension}
Let $S$ be a $D$-graded integral domain such that $\deg_D(S)\subseteq D_\R$ spans a convex cone of maximal dimension. Then  
\begin{enumerate}
	\item $\Proj^D S$ is an integral scheme, and 
	\item $\dim \Proj_DS = \dim \Spec S - \rk D$. 
\end{enumerate}
\end{lemma}
\begin{proof}
For an integral domain $S$, $(0)$ is a homogeneous prime ideal not containing $S_+$, and hence it  is the generic point of $\Proj^DS$. In particular $\Proj^DS$ is irreducible. 
Furthermore, all localizations of $S$  are again integral domains, and so is $S_{(f)}\subseteq S_f$ for $f\in\Rel^D(S)$. These subsets cover $\Proj^DS$, which is therefore an integral domain. 
 
As far as the second claim is concerned, the condition that  $\deg_D(S)\subseteq D_{\R}$ is of maximal dimension implies that $\Rel^D(S)\neq\varnothing$. Let  $f\in\Rel^D(S)$, then $\Spec S_f\subseteq \Spec S$ is a dense open subset, hence the two have the same dimension. The geometric quotient morphism $\Spec S_{f}\to \Spec S_{(f)}$ enjoys the property that
\[
\dim \Spec S_f \equ \dim \Spec S_{(f)} + \dim G_D\ ,
\]
from which the statement follows since $\dim G_D = \rk D$, and $\Spec S_{(f)} \subseteq \Proj^DS$ is a again a dense open subset.  
\end{proof}
 
\begin{remark}[Distinguished open subsets]
It is immediate that $\Proj_DS$ is isomorphic to the scheme glued together from the collection $\{D_+(f)\mid f\in \Rel(S)\}$ along the open immersions $D_+(fg)\hookrightarrow D_+(f)$ for all pairs of relevant elements $f,g$. Let $f\in\Rel^D(S)$, and let  $h\in S$ be a not necessarily relevant  homogeneous element such that $hf\neq 0$. Then $h$ defines a principal open subset in $\Spec S_f$, whose image in $D_+(f) = \Spec S_{(f)}$ is an open subset since the morphism $\Spec S_f\to \Spec S_{(f)}$ is an open map. The image in $S_{(f)}$ equals $D_+(hf)$ (note that the latter makes sense as $hf$ is relevant). 

Then we set 
\[
D_+(h) \deq \bigcup_{f\in\Rel(S)} D_+(hf)\ \subseteq \Proj_DS\ .
\]
The $D_+(h)$ are open subschemes of $\Proj_DS$ which obey the usual formal properties. 
\end{remark}

\begin{example}
Consider the standard example of the polynomial ring $k[T_1,\dots,T_r]$ $\Z$-graded by total degree. Then the collection of open subsets $\{ D_+(T_i)\mid 1\leq i\leq r\}$ is an open cover coming from relevant elements. 

However,  if we grade the polynomial ring with $\Z^r$ via the componentwise degree for instance,
then the $T_i$'s are no longer relevant. In that case the relevant monomials are precisely the ones divisible by $T_1\cdot\ldots\cdot T_r$ (and they all give rise to the same single $D_+$).  
\end{example} 
 
\begin{lemma}
Let $S$ be a $D$-graded ring, $\{f_i\mid i\in I\}$ a system of (not necessarily relevant) homogeneous generators of $S$. Then the collection $\{ D_+(f_i)\mid i\in I\}$ forms an open cover of $\Proj_DS$. 
\end{lemma} 
\begin{proof}
Given  $g\in\Rel_D(S)$ arbitrary, the images of $\{f_i\mid i\in I\}$ generate $S_g$, hence the collection $\{ D_+(f_ig)\mid i\in I\}$
forms an open cover of $D_+(g)$. Therefore, the collection $\{ D_+(f_i)\mid i\in I\}$ is an open cover of $\Proj_DS$.  
\end{proof}
 
\begin{remark}\label{rem:spectra along a line}
	With notation as above, let $d\in D$. In this case we write $S^d \deq S^{\Z d}$ and $\Rel^d(S^d)$ for $\Rel^{\Z d}(S^d)$. In any case, $(\Z d)_\R$ is one-dimensional, hence every non-zero homogeneous element of $S^d$ is relevant. In particular
	\[
	\Rel^D(S) \cap S^d \dsubseteq \Rel^d(S^d)\ . 
	\]
	Let $f\in \Rel^D(S)\cap S^d$, then the inclusion $(S^d)_f\hookrightarrow S_f$ yields an isomorphism
	\[
	(S^d)_{(f)} \simeq S_{(f)}\ ,
	\]
	since from $s/f^m\in (S_f)_0$ it follows that $\deg(s)=m\cdot \deg(f)$, hence $s\in S^d$. In particular, $\Spec S_{(f)} \simeq \Spec (S^d)_{(f)}$. 
\end{remark}

\begin{lemma}\label{lem:multigraded vs usual proj}
	Let $S$ be a $D$-graded ring, and let $d\in D$  such that $\deg_D(S^d)\subseteq \N d$. If the collection $\{ D_+(f)\mid f\in \Rel^D(S)\cap S^d\}$ is an open cover of $\Proj^DS$ then 
	\[
	\Proj^D S \simeq \Proj^\N S^d\ .
	\]
\end{lemma}
\begin{proof}
	This is an immediate consequence of Remark~\ref{rem:spectra along a line}. 
\end{proof}

\begin{remark}[Effect of regrading on spectra]\label{rmk:regrading on spectra}
Let $\delta\colon D\to D'$ be a surjective homomorphism between finitely generated abelian groups, $S$ a $D$-graded ring, $f\in S$ a $D$-homogeneous element. By Lemma~\ref{lem:relevance after regrading}  $f\in S_\delta$ is $\delta$-homogeneous (in fact it is even $\delta$-relevant provided $f$ was $D$-relevant to begin with). Then 
\[
(S_\delta)_{(f)} \deq ((S_{\delta})_f)_{0} \equ \bigoplus_{d\in\ker\delta} (S_f)_{D,d} \ ,
\]
which gives rise to the natural inclusion of (ungraded) rings
\[
 S_{(f)} \deq   (S_f)_{D,0} \hookrightarrow (S_\delta)_{(f)}\ .
\]
This in turn leads to a morphism 
\begin{equation}\label{eqn:regrading morphism of affine schemes}
\Spec  (S_\delta)_{(f)} \longrightarrow \Spec  S_{(f)}
\end{equation}
of affine schemes. 
\end{remark} 
 
\begin{example}
Let $k$ be a field. We will compute  multihomogeneous spectra of the ring $S=k[T_1,T_2]$ for various gradings. This  example serves to illustrate that regradings have a profound effect on the geometry of $\Proj^DS$ (including separatedness and dimension), and highlights the contrast to the classical $\N$-graded spectrum. 
\begin{enumerate}
	\item First let $D=\Z^2$, and grade $S$ by degree, that is set $S_{(a,b)}=k\cdot T_1^aT_2^b$ for $(a,b)\in \N^2$. Then homogeneous elements are constant multiples of monomials, and 
	\[
	\Rel_D(S) \equ \{ \alpha T_1^aT_2^b\mid a,b\geq 1\, ,\, \alpha\in k\}\ .
	\]
	As a consequence, we obtain $S_{(f)}\simeq k$ for any $f\in\Rel_D(S)$, and $\Proj_DS$ turns out to be a point. 
	\item Let $D=\Z^2$ again, but now 
	\[
	S_{(a,b)} \deq \begin{cases}
	\bigoplus_{c,d\geq 0,c+d=a} k\cdot T_1^aT_2^b & \text{ if $b=0$} \\ 0 & \text{ otherwise.}
	\end{cases}
	\] 
	This is the grading induced by the (non-surjective) homomorphism $\delta\colon \Z^2\to \Z^2\, ,\, (c,d)\mapsto (c+d,0)$. Since all degrees lie on the $a$-axis, no homogeneous element is relevant, hence $\Proj_DS=\emptyset$. 
	\item For the sake of completeness let now $D=\Z$, and consider the grading of $S$ by total degree. Essentially by definition $\Proj_DS$ is the usual $\N$-graded projective spectrum of $S$, that is, $\Proj_DS\simeq \PP^1$. The grading in this case can be realized via the surjective homomorphism $\delta\colon \Z^2\to \Z$ given by $(a,b)\mapsto a+b$. 
	 \item Last, we look at the surjective regrading $\delta\colon \Z^2\to \Z$ via $(a,b)\mapsto a-b$. As explained in \cite{BS}*{Beginning of Section 3}, we have that $\Proj_DS$ is isomorphic to the affine line with double origin. 
\end{enumerate}
\end{example}

\begin{example}\label{ex:regrading changes the spectrum}
	Let $k$ be a field, and consider the polynomial ring $k[T_1,\dots,T_r]$ with the $\Z^r$-grading induced by $\deg(T_i)=e_i\in \Z^r$. Then $\Proj_{\Z^r} k[T_1,\dots,T_r]$ is a single point. At the same time if  we consider the $\Z$-grading given  $\delta\colon \Z^r\to\Z$ obtained by the sum of the coordinates  
	then (by returning to the classical $\N$-graded $\Proj$-construction), we obtain $\Proj_\delta k[T_1,\dots,T_r]\simeq \PP^{r-1}$.  
	
	By applying the results of \cite{BS}*{Section 3}, one can see that given a surjective linear map $\delta\colon \Z^r\to \Z^m$ with $m\leq r$, the projective spectrum $\Proj_\delta k[T_1,\dots,T_r]$ is a possibly non-separated $r-m$-dimensional integral scheme. 
\end{example}

\begin{proposition}\label{prop:morphism from regrading}
	Let $D$, $D'$ be finitely generated abelian groups, $\delta\colon D\to D'$ a group homomorphism, and $S$ a $D$-graded ring. Then there exists a natural rational map  $\Proj(S_\delta)\rat \Proj(S)$ such that the diagram
	\[
	\xymatrix{
		\Proj(S_\delta) \ar@{-->}[r] \ar[d] & \Proj(S) \ar[d] \\
		\Spec( (S_\delta)_0) \ar[r] & \Spec(S_0)	
	}
	\]
	commutes. 
\end{proposition}
\begin{proof}
Let $f\in S$ be a $D$-relevant element. Then $f\in S_\phi$ is $\delta$-relevant as well. By (\ref{eqn:regrading morphism of affine schemes}) we obtain a morphism of affine schemes
$\Spec  (S_\delta)_{(f)} \longrightarrow \Spec  S_{(f)}$, which, as we let $f$ run throught all relevant elements of $S$ glue together to a morphism
\[
U \longrightarrow \Proj(S)\ ,
\] 
where $U\subseteq \Proj (S_\delta)$ is the union of all $\Spec(S_\phi)_{(f)}\subseteq \Proj (S_\phi)$, where $f\in \Rel_{D'}(S_\delta)$ was $D$-relevant in $S$ as well. 
The commutativity of the diagram follows from the commutativity of the appropriate diagram for rings.  
\end{proof}
 
\begin{example}
We return to our recurring example $S=k[T_1,T_2]$, $D=\Z^2$ the grading by monomial degree, and $\delta\colon \Z^2\to D'\deq \Z$ the homomorphism $(a,b)\mapsto a+b$.
Then 
\[
\Rel_D(S) \equ \{\alpha\cdot T_1^aT_2b\mid a,b\geq 1  \}  \ ,
\]
hence the open subset $U$ in the proof of the previous proposition is
\[
U \equ \bigcup_{f\in \Rel_D(S)} \Spec (S_\delta)_{(f)}\ .
\]
For every $f\in \Rel_D(S)$ though, $\Spec (S_\delta)_{(T_1^aT_2^b)} \equ \PP^1\setminus \{0,\infty\}$. 
\end{example}

\begin{proposition}[\cite{BS} Proposition 2.5 and 3.2]\label{prop:univ closed and finite type}
With notation as above, 
\begin{enumerate}
	\item If $S$ is noetherian, then the morphism $\Proj (S) \to \Spec (S_0)$ is universally closed and of finite type.
	\item If for every pair of points $x,y\in\Proj (S)$ there exists $f\in S$ relevant such that $x,y\in D_+(f)$, then the scheme $\Proj (S)$ is separated.
\end{enumerate}
In particular, if $S$ is noetherian and $\Spec S_0$ is a point  then $\Proj_DS$ is a complete prevariety. 
\end{proposition}

\begin{remark}
If $S$ is a finitely generated algebra over a field $k$, then it suffices to assume in Proposition~\ref{prop:univ closed and finite type} that $S_0$ is a finite-dimensional $k$-vector space. 
\end{remark}

We will be interested in showing that certain multihomogeneous projective spectra are separated, or more generally, finding large separated open subsets in them. To this end, 

\begin{proposition}[\cite{BS} Proposition 3.3]\label{prop: separated open subsets}
With notation as above, let $\{f_i\mid i\in I\}\subseteq S$ be a collection of relevant elements such that 
\[
\intt \left( \sC(f_i) \cap \sC(f_j)  \right) \neq \emptyset \ \ \text{for all $i,j\in I$}\ .
\]	
then 
\[
\bigcup_{i\in I} D_+(f_i) \dsubseteq \Proj (S)
\]
is a separated open subset. In particular, if the $D_+(f_i)$'s cover $\Proj(S)$, then the latter is separated. 
\end{proposition}

\begin{example}
If $S$ is graded by $\Z$, then $\Proj S$ is separated  if all non-zero degrees have the same sign. This shows that projective spectra of $\N$-graded rings are always separated.  
\end{example}

\begin{example}
Consider the case $S=R[T_1,\dots,T_k]$ with $R$ a commutative ring, and the grading is given on monomials via the linear map $\phi\colon \Z^k\to \Z^m$. As discussed in \cite{BS}*{Section 3}, such a grading yields a simplicial torus embedding with torus $\Spec R[\ker \phi]$. It is natural to suspect that $\Proj(S_\phi)$ is separated whenever $\im \deg \subseteq D\otimes_\Z \R$ generates  a strictly convex cone, nevertheless this is not the case.

Let $k=3$, $m=2$ and $\phi$ be given by $(1,0,0)\mapsto (1,0)$, $(0,1,0)\mapsto (1,1)$, and $(0,0,1)\mapsto (0,1)$. In search of relevant homogeneous elements we can restrict our attention to square-free monomials, of which only $T_1T_2,T_1T_3,T_2T_3$, and $T_1T_2T_3$ turn out to be relevant. Since the interior of $\sC(T_1T_2)\cap \sC(T_2T_3)$ is empty, the open subsets  $D_+(T_1T_2)\cup D_+(T_1T_3)$ or $D_+(T_2T_3)\cup D_+(T_1T_3)$ are separated, but $\Proj(S_\phi)$ itself is not.   
\end{example}

The following result describes morphisms to multihomogeneous spectra.

\begin{proposition}[\cite{BS} Proposition 4.2]\label{prop:morphisms to spectra}
Let $D$ be a finitely generated abelian group, $S$ a $D$-graded ring. In addition let $X$ be a scheme, $\cB$ a quasicoherent $D$-graded $\OO_X$-algebra, and let 
$\phi\colon S\to \Gamma(X,\cB)$ a morphism of $D$-graded rings. Write 
\[
U\deq \bigcup_{f\in S \text{ relevant}} X_{\phi(f)}\ .
\]	
Then there exists a natural morphism of schemes $r_{\cB,\phi}\colon U\to \Proj (S)$ along with a commutative diagram
\[
\xymatrix{
U \ar[d]^{r_{\cB,\phi}}   & \Spec (\cB|_U) \ar[l] \ar[r] \ar[d]^{\phi^{\#}} & \Spec(\cB) \ar[d]^{\phi^{\#}}  \\
\Proj(S) & \Spec(S) \setminus V_+ \ar[l] \ar[r] & \Spec(S)
}
\]
Here the unnamed morphisms are natural projection and inclusion morphisms. 
\end{proposition}

\begin{corollary}\label{cor:rtl map morphism}
With notation as above every point $x\in X$ has an affine open neighbourhood of the form $X_{\phi(f)}$ with $f$ relevant, then $r_{\cB,\phi}$ is a morphism. 
\end{corollary}


\begin{lemma}
	Let $D$ be a finitely generated abelian group, $\sigma\colon S_1\to S_2$ a morphism of $D$-graded rings. Then $\sigma$ induces a natural morphism 
	\[
	f_\sigma \colon \Proj(S_2) \setminus V_+(\sigma((S_1)_+)) \longrightarrow \Proj (S_1)
	\]
	such that for every relevant element $s\in (S_1)_+$ one has 
	\[
	(f_\sigma)^{-1}(D_+(s)) \equ D_+(\sigma(s))\ ,
	\]
	and the restriction $f_\sigma|_{D_+(\sigma(s))}$ coincides with the canonical morphism $\Spec((S_2)_{(\sigma(s))})\to \Spec((S_1)_{(s)})$. 
	
	If $\sigma$ is surjective then $V_+(\sigma((S_1)_+)) \equ \emptyset$, hence the morphism extends to 
	\[
	f_\sigma \colon \Proj(S_2) \longrightarrow \Proj(S_1)\ ,
	\]
	which turns out to be a closed immersion.  
\end{lemma}	 
\begin{proof}
	This is \cite{Bosch}*{Proposition 9.20 and Corollary 9.21} mutatis mutandis. 
\end{proof}

A key property of multihomogeneous spectra is their relationship to toric geometry. 

\begin{definition}
Let $R$ be a ring, $M$ a free abelian group of finite rank. A \emph{simplicial torus embedding} of the torus $T\deq \Spec R[M]$ is a $T$-equivariant open embedding
$T\hookrightarrow X$ locally given by 
\[
R[M\cap \sigma^\vee ] \hookrightarrow R[M]
\]
for a strongly convex simplicial cone $\sigma\subseteq \HHom_\Z(M,\Z)\otimes_\Z\R$. 
\end{definition}

Following \cite{BS} we do not require the torus embedding to be separated. Recall the construction of Cox \cite{Cox}: let $d\colon \Z^k\to D$ be a linear map, consider the induced $D$-grading on $S=R[T_1,\dots,T_k]$ via $\deg (T_i) = d(e_i)$.

\begin{proposition}[\cite{BS} Proposition 3.4]\label{prop:torus}
	Let $D$ be a finitely generated abelian group, $R$ an arbitrary commutative ring (with $1$), and let $S = R[T_1,\dots,T_r]$ be a $D$-graded polynomial algebra where the grading comes
	from a linear map $d\colon \Z^r\to D$ via sending the element $e_i\in \Z^r$ to $\deg(T_i)\in D$. Then $\Proj(S)$ is a possibly non-separated  simplicial torus embedding of the torus
	$\Spec(R[\ker\, d])$.  	
\end{proposition}

\begin{corollary}\label{cor:embedding into toric variety}
	Let $R$ be a commutative ring, $D$ be a finitely generated abelian group, $S$ a $D$-graded ring finitely generated over  $S_0=R$. Then $\Proj(S)$ admits a closed embedding into a simplicial torus embedding.
\end{corollary}
\begin{proof}
	The $D$-graded ring $S$ being finitely generated over $S_0=R$ is equivalent to the existence of a surjective $R$-algebra homomorphism $\phi\colon R[T_1,\dots,T_r]\to S$. We equip the polynomial algebra $R[T_1,\dots,T_r]$ with a $D$-grading by setting $\deg(T_i) \deq \deg_S(\phi(T_i))$. This way $\phi$ becomes a surjective morphism of $D$-graded rings. 
	
	The associated morphism $\phi^\#\colon \Proj(S)\to \Proj(R[T_1,\dots,T_r])$ is a closed embedding, and the latter is a simplicial torus embedding by Proposition~\ref{prop:torus}.  
\end{proof}
 
\begin{example}[\cite{BS} Example 3.8]
	If a polynomial ring  $S=R[T_1,\dots,T_r]$ is graded by $\Z$ in such a way that all indeterminates have positive degree, then $\Proj(S)$ turns out to be a weighted projective space.
\end{example}

\begin{corollary}\label{cor:torus embedding a variety}
	Let $R$ be a commutative ring, $S=R[T_1,\dots,T_r]$ considered as a $\Z^k$-graded ring as above. If $\Proj(S)$ is a separated scheme, then it is a projective variety.
\end{corollary}
\begin{proof}
As summarized in the proof of \cite{BS}*{Corollary 3.6}, since $S$ is finitely generated over $S_0$, the scheme $\Proj(S)$ is of finite type and universally closed (\emph{loc. cit.} Proposition 2.5), even a divisorial scheme (\emph{loc. cit.} Corollary 3.5). As $S$ is integral, so is $\Proj(S)$, hence we are left 
with making sure that it is separated. 

\end{proof}

\section{Ample families}

Here we review (and somewhat extend) the notion of ample families originally defined by Borelli \cite{Borelli}. As opposed to \cite{Borelli} and \cite{BS}, our notion takes different gradings into account. For this reason we give a fairly detailed account of \cite{BS}*{Section 4} with the appropriate changes. 

\begin{definition}[Ample family \cite{Borelli}]
Let $X$ be a qcqs scheme,  $\cL_1,\dots,\cL_r$ invertible sheaves on $X$. Write $\cL^d$ for $\cL_1^{\otimes d_1}\otimes\ldots\otimes \cL_r^{\otimes d_r}$ where $d=(d_1,\ldots,d_r)\in\Z^r$. We say that $\cL_1,\dots,\cL_r$ is an \emph{ample family} on $X$ if the collection of open sets $X_f$ with $f\in \Gamma(X,\cL^d)$ and 
$d\in\N^r$ forms  a basis of the topology of $X$. A qcqs scheme is called \emph{divisorial} if it supports an ample family.  	
If $\cL_1,\dots,\cL_r$ is an ample family on $X$, then we write $S_\cL\deq \oplus_{d\in \N^r}\Gamma(X,\cL^d)$.
\end{definition} 
  
\begin{remark}
Both quasiprojective varieties and locally $\QQ$-factorial varieties are divisorial. More importantly, if a $D$-graded ring $S$ is finitely generated over $S_0$, then $\Proj (S)$ is divisorial (see \cite{BS}*{Corollary 3.5}), i.e. it admits an ample family.  
\end{remark}

\begin{remark}
Let $X$ be a qcqs scheme,  $\cL_1,\dots,\cL_r$ be invertible sheaves  on $X$, and  let in addition $\delta\colon \Z^r\to D$ be a surjective homomorphism of  abelian groups. We write $S_{\cL,\delta}$ for the  ring $S_\cL$ with grading induced by $\delta$, that is
\[
\left(S_{\cL,\delta}\right)_d \equ \bigoplus_{a\in \delta^{-1}(d)} \Gamma(X,\cL^a)\ \ \ \  \text{for every $d\in D$.}
\]
\end{remark}

\begin{proposition}\label{prop:ample family topology}
Let $C$ be a qcqs scheme, $\cL_1,\dots,\cL_r$ invertible sheaves on $X$, let $\delta\colon \Z^r\to D$ be a surjective homomorphism of abelian groups.  Then the following are equivalent.
\begin{enumerate}
\item The open sets $X_f$ with $f\in (S_{\cL,\delta})_d$ and $d\in D$ form a basis of the (Zariski) topology of $X$.
\item For every point $x\in$ there exists $d\in D$ and $f\in (S_{\cL,\delta})_d$ such that $X_f$ is an affine neighbourhood of $x$. 
\item For every point $x\in X$ there exists a $\QQ$-basis $d_1,\dots,d_r\in D$ and homogeneous elements $f_i\in (S_{\cL,\delta})_{d_i}$ such that the $X_{f_i}$'s are affine neighbourhoods of $x$.
\end{enumerate} 
\end{proposition}
\begin{proof}
The proof of \cite{BS}*{Proposition 1.1} goes through verbatim. 
\end{proof}

Let $X$ be a scheme, $\cB$ a quasi-coherent graded $\OO_X$-algebra, and $\phi\colon S\to \Gamma(X,\cB)$ a graded homomorphism. As explained in \cite{BS}*{Proposition 4.2}, upon setting 
\[
U \deq U_{\cB,\phi} \deq \bigcup_{f\in S \text{ relevant}} X_{\phi(f)}\ ,
\]
one obtains a natural morphism $r_{\cB,\phi}\colon U\to \Proj (S)$, that is, a natural rational map $r_{\cB,\phi}\colon X\dashrightarrow \Proj (S)$. 

\begin{remark}
Let $X$ be a qcqs scheme, $\cL_1,\dots,\cL_r$ an ample family on $X$, $\delta\colon \Z^r\to D$ a surjective morphism of abelian groups, $S$ a $D$-graded ring. Consider the $D$-graded $\OO_X$-algebra
\[
\cB \deq \cB_{\cL,\delta} \deq \bigoplus_{d\in D} \left( \bigoplus_{a\in \delta^{-1}(d)} \cL^{a}\right) \ . 
\]
Then $\cB$ is quasicoherent. By the ampleness of the family $\cL$, if the graded homomorphism $\delta\colon S\to \Gamma(X,\cB) \equ S_{\cL,\delta}$ is in addition surjective, then 
\[
U_{\cB,\delta} \equ  \bigcup_{f\in \Rel(S)} X_{\delta(f)} \equ X\ ,
\]
in particular the rational map $r_{\cB,\delta}$ of \cite{BS}*{Proposition 4.2} is an actual  morphism $X\to \Proj S$. 
\end{remark}

The connection between  multihomogeneous spectra and ample families goes as follows.

\begin{example}\label{ex:canonical rtl map}
Let $X$ be a qcqs scheme, $\cL_1,\dots,\cL_r$ invertible sheaves on $X$, and let $\delta\colon \Z^r\to D$ be a surjective group homomorphism. With the notation of Proposition~\ref{prop:morphisms to spectra} let $S=S_{\cL,\delta}$ and $\cB=\cB_{\cL,\delta}$. Since $\Gamma(X,\cB_{\cL,\delta})=S_{\cL,\delta}$, we can take $\phi\colon S\to \Gamma(X,\cB)$ to be the identity map. In this context we will identify $f$ with $\phi(f)$. 

Via Proposition~\ref{prop:morphisms to spectra} these choices give rise to a canonical rational map $X\rat\Proj S_{\cL,\delta}$. By Corollary~\ref{cor:rtl map morphism} this rational map is a morphism if for every point $x\in X$ there exists a relevant section $f\in S_{\cL,\delta}$ such that $x\in X_f$. 
\end{example}

\begin{thm}[cf. \cite{BS} Theorem 4.4]\label{thm:ample families and morphisms}
Let $X$ be a qcqs scheme, $\cL_1,\dots,\cL_r$ be a collection of invertible sheaves on $X$, and let $\delta\colon \Z^r\to D$ be a surjective group homomorphism.
Then the following are equivalent.
\begin{itemize}
	\item[(1)] $\cL_1,\dots,\cL_r$ form an ample family. 
	\item[(2)] The natural rational map $r\colon X\dashrightarrow \Proj(S_{\cL,\delta})$ is everywhere defined, and it is an open embedding. 
\end{itemize}
\end{thm}

We will recite the proof of \cite{BS} with the necessary modifications. 

\begin{proof}
(1)$\Rightarrow$(2) Let $x\in X$ be an arbitrary point. According to Proposition~\ref{prop:ample family topology}, $D$ possesses a $\QQ$-basis $d_1,\dots,d_m$ (here $m\leq r$) such that there exist  global sections $f_i\in (S_{\cL,\delta})_{d_i}$ ($1\leq i\leq m$) giving rise to affine neighbourhoods $X_{f_i}$ of $x$. Observe that $f\deq f_1\cdot\ldots\cdot f_m\in S_{\cL,\delta}$ is relevant by Remark~\ref{rem:product is relevant}. It follows from Example~\ref{ex:canonical rtl map} that $X\rat \Proj S_{\cL,\delta}$ is defined at $x$, and since $x$ was arbitrary, it is a morphism. 

To verify that the canonical morphism $X\to \Proj S_{\cL,\delta}$ is an open immersion, take $f\in \Rel(S_{\cL,\delta})$ such that $X_f=X_{\phi(f)}$ is affine.

Then  the canonical map $\Gamma(X,\cB_{\cL,\delta})_f\to \Gamma(X_f,\cB)$ is bijective, hence it induces an isomorphism $X_f\to D_+(f)$. These isomorphisms glue together to an open embedding $X\to \Proj S_{\cL,\delta}$.  \\
(2)$\Rightarrow$(1) The implication follows from the definition of an ample family via Proposition~\ref{prop:ample family topology}. 
\end{proof}

\begin{corollary}[\cite{BS} Corollary 4.5 and 4.6]\label{cor:isom vs septd}
Let $X$ be a qcqs scheme, $\cL_1,\dots,\cL_r$ an ample family on $X$,$\delta\colon \Z^r\to D$ a surjective homomorphism. Then the following hold. 
\begin{enumerate}
	\item If $f\in \Rel(S_{\cL,\delta})$  then $X_f$ is quasi-affine. 
	\item The following are equivalent.
	\begin{enumerate}
		\item The open embedding $X\to \Proj S_{\cL,\delta}$ is an isomorphism.
		\item For every $f\in \Rel(S_{\cL,\delta})$  the open subset $X_f\subseteq X$ is affine. 
	\end{enumerate}
	If in addition  the affine hull of $X$ is proper then the above are equivalent to $\Proj S_{\cL,\delta}$ being separated.  
\end{enumerate}	
\end{corollary}
\begin{proof}
The proofs of \cite{BS}*{Corollary 4.5 and 4.6} work here verbatim. 
\end{proof}

\begin{thm}[cf. \cite{BS} Theorem 4.4]\label{thm:ample families and morphisms when fin gen}
	Let $X$ be a qcqs scheme, $\cL_1,\dots,\cL_r$ an ample family  on $X$, and let $\delta\colon \Z^r\to D$ be a surjective group homomorphism.  Assume in addition that $X$ is of finite type over a noetherian ring $R$. 
	
	Then  there exists a finite collection of sections $f_i\in \Gamma(X,\cL^{d_i})$ for an index set $i\in I$, and there exists a $D$-graded polynomial algebra $A \deq R[T_i\mid i\in I]$ such that the rational map $r\colon X\dashrightarrow \Proj(A)$ induced by $T_i\mapsto f_i$ for all $i\in I$ is everywhere defined and an embedding, which becomes a closed embedding if $X$ is proper.  
 \end{thm}

\begin{proof} 
	In order to keep track of the regrading and since we will need it later on,  we  give a complete proof (which is a minor modification of \cite{BS}*{Theorem 4.4}). To  this end, let $X$ be a qcqs scheme of finite type over a noetherian ring $R$, and let  $\cL_1,\dots,\cL_r$ be an ample family. 
	
	By definition the sets $X_f$ where $f\in S_\cL$ generates the Zariski topology of $X$. The latter is quasi-compact, hence there must exist a finite collection of $\delta$-relevant elements $f_1,\dots,f_k\in S_\cL$ such that $X_{f_1},\dots,X_{f_k}$ form an affine open cover of $X$.   
	
	Consequently, we have $X_{f_i} \simeq \Spec(A_i)$,  or, equivalently, $\Gamma(X_{f_i},\OO_X) \simeq A_i$ as $R$-algebras, where the $A_i$'s are finitely generated algebras over $R$. Let $h_{i,1},\ldots,h_{i,m_i}$ be a generating set of $A_i$ over $R$, then for every $1\leq i\leq k$ and every $1\leq j\leq m_i$ there exist $n_{i,j}\in\N$ such that $f_i^{n_i}h_{i,j} \in S_{\cL}$.
	
	Set $J \deq \{ (i,j)\in \N^2 \mid 1\leq i\leq k\, ,\, 1\leq j\leq m_i \}$ and  $g_{i,j} \deq f_i^{n_{i,j}}h_{i,j}$
	for every $(i,j)\in J$. Consider the polynomial ring $A=R[T_{(i,j)}\mid (i,j)\in J]$ graded by the homomorphism
	\begin{eqnarray*}
		\delta \colon  \Z^{J} & \to & D \\
		e_{i,j} & \mapsto & \delta(\deg g_{i,j})
	\end{eqnarray*}  
	so that the natural homomorphism $A\to S_{\cL,\delta}$ given by $T_{(i,j)}\mapsto g_{i,j}$ is $D$-homogeneous. 
	
	Observe the homomorphism $A\to S_{\cL,\delta}$ above itself  might not be,  the induced homomorphisms of rings $A_{(T_{(i,j)})} \to (S_\cL)_{(g_{i,j})}$ are all  surjective, hence  the corresponding morphisms of schemes $X_{g_{i,j}}\to D_+(T_{i,j})$ are closed embeddings. As the $X_{(g_{(i,j)})}$ form an open cover of $X$, we obtain an embedding $X\hookrightarrow \Proj_DA$. If $X$ is proper then this is a closed embedding.  
\end{proof}

\begin{remark}
Note that in verifying that the morphism $X\to\Proj_DA$  one cannot directly argue by the surjectivity of  $A\to S_{\cL,\delta}$, since $X\not\simeq \Proj_D S_{\cL,\delta}$ in general. 
\end{remark} 

\begin{remark}\label{rem:one ample is enough}
If the ample family $\cL_1,\dots,\cL_r$ contains an actual ample invertible sheaf $\cA$, then it can be arranged that all the sections $g_{(i,j)}$ be global sections of a certain power $\cA^{\otimes m}$. 
\end{remark}
 
\begin{example}[Non-separability for Cox rings of Mori Dream spaces]\label{ex:Proj Cox nonsep}
	The concrete case we are interested in is the following. Let $X$ be a projective variety, $\cL_1,\dots,\cL_r$ a finite collection of line bundles on $X$. We consider the $N^1(X)$-graded multisection ring $S_{\cL,\delta}$, where $\delta\colon \Z \cL_1\oplus\ldots\oplus\Z \cL_r \to N^1(X)$ is the numerical equivalence class map. 

    Let us now assume that $\Pic(X)$ is free of rank $\rho$ and coincides with $\Cl(X)$.  If the collection of line bundles 
    $\cL_1,\dots,\cL_\rho$ forms a basis of $\Pic(X)$, then one defines the Cox ring of $X$ by 
    \[
    \Cox(X;\cL_1,\dots,\cL_r) \deq 
    \bigoplus_{(m_1,\dots,m_\rho)\in\Z^\rho} H^0(X,\cL_1^{\otimes m_1}\otimes\ldots\otimes\cL_\rho^{\otimes m_\rho})\ .
    \]
    
    It is very natural to study $\Proj \Cox(D;\cL_1,\dots,\cL_r)$ via the natural $\Z^\rho$-grading, or, which is equivalent in this case, via $N^1(X)$.  
    
    A surprising obstacle is that  $\Proj_{N^1(X)} \Cox(X)$ is typically not separated, not even in the simplest cases. Therefore,  the open immersion $X\hookrightarrow \Proj^{N^1(X)} \Cox(X)$ is \emph{not} an isomorphism. This is so even if $\Cox(X)$ is finitely generated, i.e. when $X$ is a Mori dream space. 
    
    For a concrete example let $X$ be the blow-up of $\PP^2$ at a point $P$ with exceptional divisor $E$, and $L\subseteq X$ the strict transform of a line in $\PP^2$ not containing $P$. Then the divisor $D\deq L+E$ is relevant, but its complement, which is isomorphic to $\mathbb{A}^2\setminus \{(0,0)\}$, is not affine. Therefore $\Proj^{N^1(X)} \Cox(X)$ is not separated by  Corollary~\ref{cor:isom vs septd}. 
\end{example} 
 
\begin{remark}\label{rem:Goodman}
A result of Goodman, \cite{Goodman}*{Theorem 1},  relates divisors with affine complements to birational models, and so to the minimal model program. 
Goodman's theorem says the following: let $V$ be a complete variety, $U\subseteq V$ an open subset. Then $U$ is affine if and only if $V$ has a closed subscheme $F$ not meeting $U$ such that the inverse image of $X\backslash U$ in the blowing-up $\tilde{V}$ of $V$ along $F$ is the support of an effective ample Cartier divisor. 	

Observe that under the circumstances of \cite{KKL}*{Theorem 5.4} a $D$-MMP does not touch the complement of the divisor $D$.
\end{remark} 

The next result stands in sharp contrast to Example~\ref{ex:Proj Cox nonsep}

\begin{corollary}\label{cor:Proj of the ample cone}
Let $X$ be  a projective variety,  $\cL_1,\dots,\cL_r$ a collection of ample line bundles on $X$. Then $\Proj^{\Z^r} S_\cL\simeq X$, in particular it is separated.  
\end{corollary}
\begin{proof}
Any element of the ring $S_\cL$ is an ample divisor, and complements of ample divisors are affine. The claim therefore follows from Corollary~\ref{cor:isom vs septd}.
\end{proof}

\begin{corollary}\label{cor:ample local Cox ring}
Let $X$ be a projective variety, $\mC$ a finitely generated local Cox ring contained in the ample cone. Then $\Proj^{N^1(X)} \mC\simeq X$. 	
\end{corollary}

We end the section by constructing closed embeddings into simplicial toric varieties.

\begin{theorem}\label{thm:closed embedding into separable}
Let $X$ be a scheme of finite type over a noetherian ring $R$, and $\cL_1,\dots,\cL_r$ a collection of invertible sheaves on $X$, and let $\delta\colon \Z^r\to D$ be a surjective homomorphism of abelian groups. Write $\rho\deq \rk\, D$. 

Assume that there exist members of the family $\cA_1,\dots,\cA_\rho$ with the following list of properties.
\begin{enumerate}
	\item $\delta(\cA_1),\dots,\delta(\cA_\rho)$ form a $\QQ$-basis of $D$. 
	\item For every $1\leq i\leq \rho$ there exist global sections $s_{i,1},\dots,s_{i,j_i}\in \Gamma(X,\cA_i)$ such that for any given $1\leq i\leq \rho$  there exists a finite  affine open cover $X_1,\dots,X_m$ of $X$ such that for every $1\leq l\leq m$  the images of the $s_{(i,u)}$'s in $\Gamma(X_l,\cO_{X})$ generate the ring $\Gamma(X_l,\cO_X)$. 
\end{enumerate}	
Then there exists an embedding $X\hookrightarrow \Proj_D R[T]$ into the multigraded spectrum of a multivariate polynomial ring where the latter is separated. This embedding is closed if $X$ is proper. 
\end{theorem} 
\begin{proof}
Consider the set of global sections
\[
t_I \deq \prod_{1\leq i\leq \rho} s_{(i,l_i)}\ \in \ \Gamma(X,\cA_1\otimes\ldots\otimes\cA_\rho)\ ,
\]	
where $1\leq l_i\leq j_i$, and assume that $I$ runs through all possible combinations of indices. By condition $(2)$ and 
the equality $D_+{(t_I)} = D_+(s_{(1,l_1)})\cap\ldots\cap D_+(s_{(\rho,l_\rho)})$, the open sets $D_+(t_I)$ for all values of $I$ form an affine open cover of $X$.  	

As in the proof of Theorem~\ref{thm:ample families and morphisms when fin gen}, we work with the polynomial ring $R[T] \deq R[T_I\mid \text{all possible values of $I$}]$, and observe that by grading $R[T]$  by the homomorphism
	\begin{eqnarray*}
		\delta \colon  \Z^{J} & \to & D \\
		e_{i,j} & \mapsto & \delta(\deg g_{i,j})
	\end{eqnarray*}  
the natural homomorphism $R[T]\to S_{\cL,\delta}$ sending  $T_{(i,j)}\mapsto g_{i,j}$ is surjective and $D$-homogeneous. 
The conditions on the $s_{i,j}$'s  for a given  $1\leq i\leq \rho$ yield that the morphism of schemes $X_{t_I}\to D_+(T_I)$
is a closed embedding, hence the resulting morphism $X\hookrightarrow \Proj_D R[T]$ will be an embedding (again, closed if $X$ is proper over $\Spec R$). 

We are left with verifying the separability of $\Proj_DR[T]$. By construction
\[
\deg(T_I) \equ \delta(\deg(t_I)) \equ \delta(s_{1,l_1}) + \ldots +  \delta(s_{\rho,l_\rho}) \ .
\]
Because $\delta(\cA_1),\dots,\delta(\cA_\rho)$ form a $\QQ$-basis of $D$ by condition $(1)$, the interior $\sC(T_I)\subseteq D_\R$ is non-empty, hence $T_I\in\Rel_D(R[T])$. At the same time 
\[
\sC(T_I) \equ \text{closed convex hull of $\delta(\cA_1),\dots,\delta(\cA_\rho)$} \dsubseteq D_\R
\]
independetly of the choice of $I$, and so $\intt\, \sC(T_I)\cap \sC(T_{I'}) \neq \emptyset$. According to Proposition~\ref{prop: separated open subsets} $\cup_{I} D_+(T_I)\subseteq \Proj_DR[T]$ is a separated open subset, but the two are equal since the $D_+(T_I)$ are known to be an open covering of $\Proj_DR[T]$. 
\end{proof}

\section{Gen divisorial rings and the Mori chamber decomposition in the local setting}

Finite generation of multigraded divisorial rings is a central question in algebraic geometry which lies  at the heart of Mori's  minimal model program. It is a fundamental problem in this regard  that while finite generation of  section rings behaves well with respect to linear equivalence, this is false  for numerical equivalence of divisors. This is the source of the relevance for gen divisors, originally defined in \cite{KKL}*{Definition 4.7}. This section is  expository, mostly what it does is to recall the relevant material from \cite{KKL}. 

\begin{example}[Finite generation is not invariant with respect to numerical equivalence]\label{eg:fin gen not invariant num}
A simple example of this phenomenon is shown in \cite{PAGII}*{Example 10.3.3}, which we now recall. Let $C$ be a smooth elliptic curve, and let $A$ be an ample divisor of degree one on $C$. Consider the smooth projective surface  $X :=\PP_C(\OO_C\oplus \OO_C(A))$ realized as  a projective bundle with projection morphism  $p \colon X\to C$. 

Let $P_1, P_2$ be  divisors on $C$ such that  $P_1$ is torsion, and $P_2$ is non-torsion of  degree $0$. Consider $L_i:=  \OO_X(1)\otimes p^*\OO_{C}(P_i)$. Then $L_1$ and $L_2$ are numerically equivalent nef and big line bundles  for which  $\sB(L_1)$ is empty while $\sB(L_2)$ is not, in particular,  $R(X, L_1)$ is finitely generated while $R(X, L_2)$ is not by the Zariski--Wilson theorem \cite{PAGI}*{Theorem 2.3.15}
\end{example}

\begin{definition}\label{defn:gen}
	Let X be a $\QQ$-factorial projective variety. We say that a divisor $\mD \in \rm{Div}_{\QQ}(X)$ is \emph{gen} if for all $\QQ$-Cartier $\QQ$-divisors $\mD' \equiv \mD$, the section ring $R(X,\mD')$ is finitely generated.
\end{definition} 

\begin{remark}
As observed in \cite{KKL}, there are three main sources of gen divisors.
\begin{itemize}
	\item[(i)] Ample $\QQ$-divisors on arbitrary varieties.
	\item[(ii)] Adjoint divisors $K_X + \Delta +\cA$ on normal projective varieties, where $\Delta$ is an effective $\QQ$-divisor and $\cA$ is an ample $\QQ$-divisor on $X$, and	the pair $(X, \Delta)$ is klt. 
	\item[(iii)] Arbitrary divisors with finitely generated section rings on varieties with   $\Pic (X)_{\QQ} = N^1(X)_{\QQ}$. 
\end{itemize}
Note that the latter class includes Mori Dream Spaces. 
\end{remark}
 
\begin{remark}
In the other direction, a clear obstruction for a $\QQ$-divisor $\mD$ to be gen is  if $\vol{X}{\mD}$ is an irrational number \cites{AKL,BN,KLM_volumes}.
\end{remark}

	\begin{definition} \cite{HK}*{2.12} Let $C \subset {\rm NE}^1(X)$ be the affine hull of finitely many effective
		divisors. We say that $C$ is a Mori dream region provided the following hold:
		\begin{enumerate}
			\item there exists a finite collection of birational contractions $f_i : X \dashrightarrow Y_i$ such
			that $C_i := C \cap f^*({\rm Nef}(Yi)) \times \rm{ ex}(f_i)$ is the affine hull of finitely many effective
			divisors;
			\item $C$ is the union of the $C_i$; and
			\item any line bundle in $(f_i)_*(C_i) \cap \Nef(Y_i)$ is semi-ample.
			
		\end{enumerate}
	\end{definition}
	
	The following result of \cite{HK} is the starting point of our main result. The authors assume that $\rm{Pic}(X)_{\QQ} = N^1(X)_{\QQ}$. With the extra notion of gen divisors we are to generalize the result to a larger class of varieties.
	
	\begin{theorem} \cite{HK}*{2.13} Let X be a normal projective variety with  $\rm{Pic}(X)_{\QQ} = N^1(X)_{\QQ}$, and let $C \subset {\rm N}^1(X)$ be
		a rational polyhedral cone (i.e., the affine hull of the classes of finitely many line
		bundles). Then $C \cap {\rm NE}^1(X)$
		is a Mori dream region if and only if there are generators
		$\cL_1,...,\cL_r$ of $C$ such that ${\rm R}(X; \cL_1,\ldots,\cL_r)$ is finitely generated.
	\end{theorem}

\begin{example}[Non-example for Mori dream regions]
As described  in \cite{KKL}*{Example 4.8}, the surface  in Example~\ref{eg:fin gen not invariant num} provides an example which is problematic from the point of view of Mori dream regions (cf. \cite{HK}*{Proposition 2.13}).  One can show that there exist a big divisor $\mD$ and an ample divisor $\cA$ on $X$ such that the ring $R(X;\mD,\cA)$ is finitely generated, the divisor $L_1$ belongs to the interior of the cone $\Supp\mathfrak R=\R_+D+\R_+\cA$, and no divisor in the cone $\R_+D+\R_+L_1\subseteq\Supp\mathfrak R$ is gen. In particular, we cannot perform the MMP for $D$.	
\end{example}

	Our first aim is to state a general version of the previous Theorem, with  the notations introduced in \cite{KKL}. The additional input will come from Theorem \cite{KKL}*{Theorem 5.4}. 
	
	Let us denote by $\pi$ the natural projection $$\pi:{\rm Div}(X) \to {\rm N}^1(X).$$
	One first observation is that the preimage of a strictly convex cone in ${\rm N}^1(X)$ might not be strictly convex. And in general it will not be. Hence we will define the main objects starting from ${\rm Div}(X)$.

	\begin{theorem} \label{chambers}
		Let X be a $\QQ$-factorial projective variety. Let $\mC$ be a rational polyhedral cone in ${\rm Div}(X)$ spanned by the divisors $\mD_1, \ldots, \mD_k$. Assume that $\mathfrak{R} =R(X; \mD_1, \ldots, \mD_k)$ is a finitely generated ring, that $C = \pi({\rm Supp}(\mathfrak{R}))$ spans ${\rm N}^1(X)$ and that ${\rm Supp}(\mathfrak{R})$ contains an ample divisor. Let us also assume that every divisor in the interior of $({\rm Supp})\mathfrak{R}$ is gen. Then we have the following.
		\begin{enumerate}
			
			\item There is a finite decomposition $${\rm Supp}\mathfrak{R}= \bigsqcup \mathfrak{N}_i,$$ that induces a finte decomposition of  $C$, given by $$C= \bigsqcup \pi(\mathfrak{N}_i).$$

			\item Each $\overline{\mathfrak{N}}_i$ and $\overline{\pi(\mathfrak{N}}_i)$ are rational polyhedral cones.
			
			\item for each $i$ there exists a $\mathbb{Q}$-factorial projective variety $X_i$ and a birational contraction $\varphi_i: X \dashrightarrow X_i$, such that for every $D$ in $\mathfrak{N}_i$, then $\varphi_i$ is $D$-negative and $\varphi_{i,*}D$ is semi-ample.
			
			\item for each $i$ we have that $$\mathfrak{N}_i = \pi^{-1}\left((\varphi^*{\rm Nef}(X_i)) \times {\rm ex}\varphi_i\right) \cap {\rm Supp}\mathfrak{R}.$$
			
		\end{enumerate}
	\end{theorem}

	\begin{proof}
		The only difference from  \cite{KKL}*{Theorem 5.4} is statement $(4)$. This is the key to generalize \cite{HK}*{Proposition 1.11}. Obviously by $(3)$, $\varphi_{*}(\pi \mathfrak{N}_i) \subseteq {\rm Nef}(X_i)$, hence we only need to make sure that there is no other element in ${\rm Supp}\mathfrak{R}$ mapping to ${\rm Nef}(X_i)$. 
		
		Let us consider a divisor $\overline{\mD} \in {\rm Supp}\mathfrak{R} \backslash \mathfrak{N}_i$. Then, by condition $(\natural)$ in \cite{KKL}*{Theorem 5.4} we know that there is a hyperplane $\mathcal{H}$ containing a codimension one face of $\mathfrak{N}_i$, such that $\overline{\mD}$ and $\mathfrak{N}_i$ lie in the opposite side with respect to $\mathcal{H}$.
		
		In particular by \cite{KKL}*{Theorem 5.2}, $\overline{\mD}$ does not map to ${\rm Nef}(X_i)$.
		
	\end{proof}


\section{Ample families and embeddings into toric varieties}

In this section we prove our main result.  Given a $\QQ$-factorial normal projective variety $X$, and a finitely generated section ring as in Theorem \ref{chambers}, we will first describe an embedding $X \subseteq W$, a simplicial projective toric variety. We will prove that every small $\QQ$-factorial modification (SQM) of the MMP for $W$ induces a SQM for $X$. Conversely, SQM's coming from Mori dream regions are induced from toric SQM's on $W$. 

In order to achieve such a result we have modified  some of the main results in \cite{BS} in Section 4. Again, the results of \cite{BS}*{Section 4} are not  directly applicable, since the geometrically relevant grading (by the N\'eron--Severi group) does not make the collection of line bundles we consider into an ample family.

Let $\mC\subseteq \Div(X)$ be a finitely generated cone containing an ample divisor, and let $\mC'\deq \Supp(\mC)$. According to \cite{ELMNP} (see also \cite{KKL}*{Theorem 3.2}), the cone $\mC'$ is also finitely generated, hence so is $\Cox(\mC')$ by Gordan's lemma. Let  $\mD_1,\dots,\mD_k$ be a set of generators  of $\Cox(\mC')$, then the $\mD_i$'s also span the cone $\mC'$. 

Observe that  $\mC'$ comes equipped with two different gradings:
\begin{enumerate}
	\item $\N \mD_1\oplus\ldots\oplus \N \mD_k$, where $\mD_1,\dots,\mD_k$ are effective divisors (that is, homogeneous elements) generating $\mC'$,
	\item the N\'eron--Severi group $N^1(X)$. 
\end{enumerate}

\begin{lemma}
Consider a homogeneous element $f\in\mC'$ which is relevant with respect to $N^1(X)$. Then $f$ is a big divisor. 	
\end{lemma}	
\begin{proof}
As $f$ is relevant, its cone $\sC(f)\subseteq G_2\otimes_\Z \R = N^1(X)_\R$ is a closed convex cone of maximal dimension. By construction $\sC(f)\subseteq \overline{\Eff}(X)_2$, therefore 
$\intt \sC(f) \subseteq \intt \overline{\Eff}(X)_2$. According to Lemma~\ref*{lem:degree in the interior} $\deg_{G_2}(f)\in \intt \sC(f)$, so we are done. 
\end{proof}

In \cite{BS} the authors focus on the grading $\N \mD_1\oplus\ldots\oplus \N \mD_k$. In this section we will always use the N\'eron--Severi grading and implicitly compare it with the first grading by the following Lemma.

\begin{lemma}\label{lem: C prime is an ample family}
	With notation as above, if $\mC'$ contains an ample divisor, then the collection $\OO_X(\mD_1),\ldots, \OO_X(\mD_k)$ of line bundles forms an ample family.
\end{lemma}
\begin{proof}
	By assumption $\mC'$ contains an (integral) ample divisor $\cA$. Since the divisors $\mD_1,\ldots,\mD_k$ generate $\mC'$, there exist natural numbers $m_1,\ldots,m_k$ such that $\cA=\sum_{i=1}^{k}m_i\mD_i$. 
	
	The open sets $X_s$ with $s\in \Gamma(X,\OO_X(m\cA))=\Gamma(X,\OO_X(m(\sum_{i=1}^{k}m_i\mD_i)))$ and $m$ a positive integer form a basis of the Zariski topology of $X$, hence so does the collection
	\[
	\{ X_s \mid s\in\Gamma(X,\OO_X(d_1\mD_1+\ldots+d_k\mD_k))\, ,\, (d_1,\ldots,d_k)\in\N^k\}\ .
	\]
	Therefore $\OO_X(\mD_1),\ldots, \OO_X(\mD_k)$ constitute an ample family, as claimed. 
\end{proof}

\begin{rmk}
With a slight abuse of notation we will call a collection of divisors $\mD_1,\dots,\mD_k$ an ample family, if the corresponding set of invertible sheaves is. 
\end{rmk}  

\begin{lemma} \label{lemma:proj equality}
Let $X$ be a $\QQ$-factorial normal projective variety such that $\mC\subseteq\Div(X)$ a finitely generated cone of maximal dimension and   $\Cox(\mC')$ is a finitely generated $N^1(X)$-graded ring. Let $\mC^A$ be a maximal dimensional cone in $\mC'$ generated by finitely many ample divisors $\cA_1, \ldots ,\cA_k$. Let  $\cA=\cA_1 + \ldots + \cA_k$ be the ample line bundle obtained as a sum of such divisors, with $\Gamma(X,\cA)\subseteq \Cox(\mC^A)$. Then 
\[
\Proj^\N {\rm R}(X, \cA) \simeq \Proj^{N^1(X)}\Cox(\mC^A)\ .
\]
\end{lemma}
\begin{proof}
Notice that the line bundle is in the interior of the cone generated by  $\cA_1, \ldots ,\cA_k$.  The statement then follows from 
\cite{MU}*{Theorem 3.4} as $\cA$ is the determinant of the vector bundle $\cO_X(\cA_1)\oplus \ldots \oplus \cO_X(\cA_k)$.

\end{proof}

\begin{remark}
Alternatively, in the proof of Lemma \ref{lemma:proj equality}, one can use Corollary~\ref{cor:ample local Cox ring}, because both sides are isomorphic to $X$.  However the previous proof will render the relation to the toric situation clearer.
\end{remark}
   
\begin{remark}
Observe that  Lemma~\ref{lemma:proj equality} relies only on the algebraic structure of $\Cox(\mC')$, and hence  is independent of the birational model of $X$ we consider. Therefore, the Lemma applies to any chamber in the decomposition of the movable cone, where we simply need $\mD$ to be any divisor not contained in a wall. This is why we will always require $\mD$ to be a general member of the movable cone.
\end{remark}

\begin{theorem}\label{thm:embedding into simplicial toric variety}
Let $X$ be a $\QQ$-factorial normal projective variety, $\mC\subseteq\Div(X)$ a finitely generated cone of maximal dimension. Consider the finitely generated  $N^1(X)$-graded ring  $\Cox(\mC')$, and  let $\mD_1,\dots,\mD_r$ be a finite system of generators. Let $\mC^A$ be a maximal dimensional cone in $\mC'$ generated by finitely many ample divisors $\cA_1, \ldots \cA_k$, where $\Cox(\mC^A)\subseteq \Cox(\mC')$. 
Then there exist a finite set of global sections of $\mC^A$ which  yield a  closed embedding  $\iota\colon X\hookrightarrow Z$ into a simplicial toric variety. 
\end{theorem}

\begin{proof}
The statement follows immediately from  Theorem~\ref{thm:closed embedding into separable}, once we verify that its conditions hold in our setting. 

According to  Lemma~\ref{lem: C prime is an ample family}, the family of invertible sheaves $\OO_X(\mD_1),\ldots, \OO_X(\mD_r)$ is ample. We will consider $S_{\OO_X(\mD_1),\dots,\OO_X(\mD_r)}$ with the grading obtained from the numerical equivalence class map $\delta\colon \Z^r\to N^1(X)$. Since the cone $\mC$ is of maximal dimension, so is $\mC'$, therefore the regrading homomorphism $\delta$ is surjective.  Being a variety, $X$ is  of finite type over the base field, and therefore Theorem~\ref{thm:closed embedding into separable} applies. 

Observe that since  $\pi(\mC')\subseteq N^1(X)_\R$  spans a cone of maximal dimension  containing the class of an ample line bundle,  there exist  ample line bundles $\cA_1, \ldots \cA_k$ whose sections satisfy the conditions required in Theorem~\ref{thm:closed embedding into separable}.

Then Theorem~\ref{thm:closed embedding into separable} gives rise to an embedding (a closed embedding since $X$ is a projective variety)
\[
\iota \colon X \hookrightarrow Z \deq \Proj^{N^1(X)} k[T]
\]
for a multivariate polynomial ring $k[T]$. 

As $Z$ is separated by Theorem~\ref{thm:closed embedding into separable}, it is a simplicial toric variety by Corollary~\ref{cor:torus embedding a variety}. 
\end{proof}

Next, we elaborate the situation to emphasize the parallels to \cite{HK} (cf. \cite{LV}*{Section 3} as well). 
Let $\mD_1,\dots,\mD_r$ be a finite system of generators for the finitely generated $\CC$-algebra  $\Cox(\mC')$. The choice of the generators gives rise to a surjective $\CC$-algebra homomorphism  $\Phi\colon \mathbb{C}[T_1, \ldots, T_r]\to \mC'$ via $T_i\mapsto \mD_i$. 

Keeping in mind Theorem~\ref{thm:embedding into simplicial toric variety}, let $\cA$ be an ample line bundle in $\Cox(\mC')$. Then there exist $\cA_1, \ldots \cA_k$ ample divisors, generating a maximal dimentional cone in $N^1(X)$ such that  $\cA=\cA_1 + \ldots + \cA_k$.  The choice of $\cA$  induces an inclusion of rings 
 \[
 \Cox(\mC')_A \deq \bigoplus_{m\geq 0}\Cox(\mC')_{[\cA^{\otimes m}]} \subseteq \Cox(\mC^A) \subseteq \Cox(\mC')\ ,
 \]
 where $A$ stands for the numerical equivalence class of $\cA$, and each summand is homogeneous with respect to the grading value, which is the numerical equivalence class of $\cA^{\otimes m}$ in $N^1(X)$. Since both the subrings $\Cox(\mC')_A$ and $\Cox(\mC^A)$ are finitely generated, they induce a commutative diagram 
 \[
 \xymatrix{
 	R(\cA) \ar[d] \ar[r] & R(\mC^A) \ar[d] \ar[r] & \CC[T_1,\dots,T_r] \ar[d] \\ \Cox(\mC')_A \ar[r] & \Cox(\mC^A) \ar[r] &  \Cox(\mC')
 }	
 \]	
 where the rings $R(\cA), \; R(\mC^A) \subseteq \CC[T_1,\dots,T_r]$ inherit the $N^1(X)$-grading from $\CC[T_1,\dots,T_r]$, and $R(\cA)$ inherits the $\N A$-grading from $\Cox(\mC')_A$. The two gradings are compatible since the $N^1(X)$-grading induces the $\N A$-grading.  As in the case of Mori dream spaces, the surjections of $\CC$-algebras 
 \[
 R(\cA) \twoheadrightarrow \Cox(\mC')_A \quad{\rm and} \quad R(\mC^A) \twoheadrightarrow \Cox(\mC^A)
 \]
 give rise to closed immersions of schemes
 \[
 j_1 \colon X \simeq \Proj^\Z \Cox(\mC')_A \hookrightarrow \Proj^\Z R(\cA)  \equ Z ,
 \] 
where the projective spectra that occur are taken with respect to the $\Z$-grading given by $A$ and
 \[
 j_2 \colon X \simeq \Proj^{N^1(X)} \Cox(\mC^A) \hookrightarrow \Proj^{N^1(X)} R(\mC^A) \equ Z ,
 \] 
where the projective spectra that occur are taken with respect to the $N^1(X)$-grading given by $\mC^A$. 
According to Lemma~\ref{lemma:proj equality} the two inclusions are naturally isomorphic. 

\begin{proposition} \label{lemma:simplicial}
With notation as above, $\Proj^\N R(\cA)$ is a simplicial toric variety with Cox ring isomorphic to the $N^1(X)$-graded ring $\CC[T_1,\dots,T_r]$.
\end{proposition}
\begin{proof}
By 	construction the $\N$-grading on $R$ is given by
\[
R_m \equ \{ P\in \CC[T_1,\dots,T_r] \mid \deg^{N^1(X)} P = mA  \}\ ,
\]
where $A$ stands for the numerical equivalence class of the ample line bundle of our choice $\cA$, and the degree of a polynomial $P$ with respect to $N^1(X)$ is the numerical equivalence of $\Phi(P)\in N^1(X)$. A standard construction in toric geometry shows that $Z\deq \Proj R$ with the above $\N$-grading is a normal projective  toric variety. 

Since $\mC'$ is a cone of maximal dimension in $N^1(X)_\R$,  $N^1(X)$-grading on $\Cox(\mC')$ induces a surjective map $\CC[T_1,\dots,T_r] \to N^1(X)$. The kernel of this map will correspond to the character-lattice $M$ for the toric variety $Z$ in which we are embedding $X$, and the  map yields the short exact sequence 
\[
0 \lra M \lra \Div(Z) \equ \CC[T_1,\dots,T_r] \lra \Pic(Z)\equ N^1(X) \lra 0\ .
\]
Finally, Theorem~\ref{thm:embedding into simplicial toric variety} shows that the toric variety obtained above is simplicial.
\end{proof}

\begin{lemma}
With notation as above, pulling back via $j\colon X \hookrightarrow \Proj^{\Z} R(\cA)$ yields the original commutative diagram of $N^1(X)$-graded rings
\[
\xymatrix{
	R(\cA) \ar[d] \ar[r] & \CC[T_1,\dots,T_r] \ar[d] \\ \Cox(\mC')_A \ar[r] &  \Cox(\mC')
}	
\]	 
\end{lemma}
\begin{proof}
By construction the gradings of the toric divisors are induced via the grading on $X$. Hence we only need to prove that $j^*({\rm div}T_i)=D_i$ for all $1\leq i\leq r$. But this holds by the commutativity of the diagram and Lemma~\ref{lemma:proj equality}.
\end{proof}

\begin{theorem}\label{thm:embedded MMP}

Let $X$ be a $\QQ$-factorial normal projective variety, $\mC\subseteq\Div(X)$ a finitely generated cone. Then there exists an embedding $X \subseteq W$ into a quasi-smooth projective toric variety such that

\begin{enumerate}
\item The restriction $N^1(W)_\R \to N^1(X)_\R$ is an isomorphism.
\item The isomorphism of (1) induces an isomorphism $NE^1(W)\to \mC$.

Even more, if each divisor in the cone $\mC$ is gen, then

\item Every Mori chamber of $X$ is a union of finitely many Mori chambers of W.
\item For every rational contraction $f: X\dashrightarrow X'$ there is toric rational contraction $\overline f: W \dashrightarrow W'$, regular at the generic point of $X$, such that $f=\overline{f}|_X$.
\end{enumerate}
\end{theorem} 

\begin{proof}
As we will see, the proofs of (1) and (2) are  quick consequences of Theorem \ref{thm:embedding into simplicial toric variety} and the Lemmas preceding our statement. For (1), observe that  the linear span of $\mC$ generates $N^1(X)_\R$ since $\mC'$ is of maximal dimension. Second, the grading on $\Cox(\mC')$ induces a $N^1(X)$-grading on $\CC[T_1,\dots,T_r] $, where the image is of maximal dimension since  $$\CC[T_1,\dots,T_r] \to \Cox(\mC')$$ is surjective. 

Recall that  the rays of the fan of a  toric variety  correspond to the generators of its Cox ring, whose image via the previous map is given by the generators of $\mC'$. Since every effective divisor is numerically equivalent to a positive $\R$-linear combination of these divisors, we have (2).

The proofs of (3) and (4) rely on the proof of \cite{HK}*{Proposition 2.11}, \cite{BS}*{Theorem 4.4} and \cite{KKL}. Let us consider the chamber decomposition for $\mC'$ given in \cite{KKL}*{Theorem 3.2}. Let $\mD$ be a general member of ${\rm Mov}(X)\cap \mC$ and $\cA$ an ample divisor. In particular, $\mD$ belongs to a chamber of the decomposition, and being a general member we can assume that the divisor belongs to the interior of one of the chambers.  

As every divisor in the cone is gen, we know that the projective spectrum  of the section  ring is given by the corresponding Iitaka fibration. For an element outside of the ample cone, we might not know if the Proj of such an algebra is a variety, but this is ensured by \cite{KKL}*{Theorem 4.2 (3)}. This immediately gives the finite generation of each divisorial ring contained in the cone. 

Again, \cite{KKL}*{Theorem 4.2} immediately yields that two chambers sharing a common wall have compatible morphisms to the image of the fibration induced by the divisors on the wall. The ample model for the chamber containing $\mD$ is uniquely defined, as the induced map is independent of the choice of the divisor in the numerical equivalence class. In particular, since the rational contraction is exactly the Iitaka fibration for the divisorial algebra, this is induced by the toric contraction  
\[
\xymatrix{
X =  {\rm Proj}\left( \bigoplus_{m\geq 0}\Cox(\mC')_{mA}\right) \ar@{^{(}->}[r] \ar@{-->}[d]^f   & {\rm Proj}\left(\bigoplus_{m\geq 0} \mathbb{C}[T_1, \ldots, T_r]_{mA}\right) \ar@{-->}[d]  \\
X'=  {\rm Proj}\left( \bigoplus_{m\geq 0}\Cox(\mC')_{mD}\right)  \ar@{^{(}->}[r] & {\rm Proj}\left(\bigoplus_{m\geq 0} \mathbb{C}[T_1, \ldots, T_r]_{mD}\right)   
}\ .
\]
We conclude the proof by observing  that $f$ is an ample model for $\mD$ and the image of $\mC$ in $N^1(X)$ can be recovered as a finite union of cones that are preimages for such maps
$${\rm deg}(\mC) = \bigcup_f {\rm cone}\left(f^*{\rm Nef}(X'), ex(f)\right)\ .$$ 
\end{proof}


\begin{thebibliography}{99}

 

\bib{AKL}{article}{
	author={Anderson, Dave},
	author={K\"uronya, Alex},
	author={Lozovanu, Victor},
	title={Okounkov bodies of finitely generated divisors},
	journal={International Mathematics Research Notices},
	volume={132},
	date={2013},
	number={5},
	pages={1205--1221},
}



\bib{SGA_I}{book}{
	label={SGA1},
	author={Artin, M.},
	author={Bertin, J. E.},
	author={Demazure, M.},
	author={Gabriel, P.},
	author={Grothendieck, A.},
	author={Raynaud, M.},
	author={Serre, J.-P.},
	title={Sch\'{e}mas en groupes. Fasc. 1: Expos\'{e}s 1 \`a 4},
	language={French},
	series={S\'{e}minaire de G\'{e}om\'{e}trie Alg\'{e}brique de l'Institut des Hautes \'{E}tudes
		Scientifiques},
	volume={1963},
	publisher={Institut des Hautes \'{E}tudes Scientifiques, Paris},
	date={1963/1964},
	pages={v+252 pp. (not consecutively paged)},
	review={\MR{0207702}},
}

\bib{CoxRingsBook}{book}{
	author={Arzhantsev, Ivan},
	author={Derenthal, Ulrich},
	author={Hausen, J\"{u}rgen},
	author={Laface, Antonio},
	title={Cox rings},
	series={Cambridge Studies in Advanced Mathematics},
	volume={144},
	publisher={Cambridge University Press, Cambridge},
	date={2015},
	pages={viii+530},
}


\bib{BKS}{article}{
	author={Bauer, Th.},
	author={K\"{u}ronya, A.},
	author={Szemberg, T.},
	title={Zariski chambers, volumes, and stable base loci},
	journal={J. Reine Angew. Math.},
	volume={576},
	date={2004},
	pages={209--233},
	issn={0075-4102},
	review={\MR{2099205}},
	doi={10.1515/crll.2004.090},
}

\bib{Borelli}{article}{
	label={Bor},
	author={Borelli, Mario},
	title={Divisorial varieties},
	journal={Pacific J. Math.},
	volume={13},
	date={1963},
	pages={375--388},
}


\bib{BN}{article}{
	author={Borntr\"ager, Carsten},
	author={Nickel, Matthias},
	title={Algebraic volumes of divisors},
	journal={European Journal of Math.},
	volume={5},
	date={2019},
    pages={1192--1201},	
}


\bib{Bosch}{book}{
	author={Bosch, Siegfried},
	title={Algebraic geometry and commutative algebra},
	series={Universitext},
	publisher={Springer, London},
	date={2013},
	pages={x+504},
	isbn={978-1-4471-4828-9},
	isbn={978-1-4471-4829-6},
	review={\MR{3012586}},
	doi={10.1007/978-1-4471-4829-6},
}


 
 
 
\bib{BS}{article}{
	author={Brenner, Holger},
	author={Schr\"{o}er, Stefan},
	title={Ample families, multihomogeneous spectra, and algebraization of
		formal schemes},
	journal={Pacific J. Math.},
	volume={208},
	date={2003},
	number={2},
	pages={209--230},
}


\bib{CaL10}{article}{
	label={CaL},
	author={Cascini, Paolo},
	author={Lazi\'{c}, Vladimir},
	title={New outlook on the minimal model program, I},
	journal={Duke Math. J.},
	volume={161},
	date={2012},
	number={12},
	pages={2415--2467},
	doi={10.1215/00127094-1723755},
}


\bib{CT}{article}{
	author={Castravet, Ana-Maria},
	author={Tevelev, Jenia},
	title={Hilbert's 14th problem and Cox rings},
	journal={Compos. Math.},
	volume={142},
	date={2006},
	number={6},
	pages={1479--1498},
	issn={0010-437X},
	review={\MR{2278756}},
	doi={10.1112/S0010437X06002284},
}

 

\bib{CoL10}{article}{
	label={CoL},
	author={Corti, Alessio},
	author={Lazi\'{c}, Vladimir},
	title={New outlook on the minimal model program, II},
	journal={Math. Ann.},
	volume={356},
	date={2013},
	number={2},
	pages={617--633},
	doi={10.1007/s00208-012-0858-1},
}


\bib{Cox}{article}{
	author={Cox, David A.},
	title={The homogeneous coordinate ring of a toric variety},
	journal={J. Algebraic Geom.},
	volume={4},
	date={1995},
	number={1},
	pages={17--50},
}


 

\bib{DH}{article}{
	author={Dolgachev, Igor V.},
	author={Hu, Yi},
	title={Variation of geometric invariant theory quotients},
	note={With an appendix by Nicolas Ressayre},
	journal={Inst. Hautes \'{E}tudes Sci. Publ. Math.},
	number={87},
	date={1998},
	pages={5--56},
}

\bib{ELMNP}{article}{
	author={Ein, Lawrence},
	author={Lazarsfeld, Robert},
	author={Musta\c{t}\u{a}, Mircea},
	author={Nakamaye, Michael},
	author={Popa, Mihnea},
	title={Asymptotic invariants of base loci},
	language={English, with English and French summaries},
	journal={Ann. Inst. Fourier (Grenoble)},
	volume={56},
	date={2006},
	number={6},
	pages={1701--1734},
}

 
\bib{Goodman}{article}{
	author={Goodman, Jacob Eli},
	title={Affine open subsets of algebraic varieties and ample divisors},
	journal={Ann. of Math. (2)},
	volume={89},
	date={1969},
	pages={160--183},
	issn={0003-486X},
	review={\MR{242843}},
	doi={10.2307/1970814},
}

 
 

\bib{HK}{article}{
	author={Hu, Yi},
	author={Keel, Sean},
	title={Mori dream spaces and GIT},
	note={Dedicated to William Fulton on the occasion of his 60th birthday},
	journal={Michigan Math. J.},
	volume={48},
	date={2000},
	pages={331--348},
	doi={10.1307/mmj/1030132722},
}


\bib{HTsch}{article}{
	author={Hassett, Brendan},
	author={Tschinkel, Yuri},
	title={Universal torsors and Cox rings},
	conference={
		title={Arithmetic of higher-dimensional algebraic varieties},
		address={Palo Alto, CA},
		date={2002},
	},
	book={
		series={Progr. Math.},
		volume={226},
		publisher={Birkh\"{a}user Boston, Boston, MA},
	},
	date={2004},
	pages={149--173},
}

\bib{KKL}{article}{
	author={Kaloghiros, Anne-Sophie},
	author={K\"{u}ronya, Alex},
	author={Lazi\'{c}, Vladimir},
	title={Finite generation and geography of models},
	conference={
		title={Minimal models and extremal rays},
		address={Kyoto},
		date={2011},
	},
	book={
		series={Adv. Stud. Pure Math.},
		volume={70},
		publisher={Math. Soc. Japan, [Tokyo]},
	},
	date={2016},
	pages={215--245},
}
 
\bib{KSZ}{article}{
	author={Kapranov, M. M.},
	author={Sturmfels, B.},
	author={Zelevinsky, A. V.},
	title={Quotients of toric varieties},
	journal={Math. Ann.},
	volume={290},
	date={1991},
	number={4},
	pages={643--655},
}


 
 

\bib{KLM_volumes}{article}{
	author={K\"{u}ronya, Alex},
	author={Lozovanu, Victor},
	author={Maclean, Catriona},
	title={Volume functions of linear series},
	journal={Math. Ann.},
	volume={356},
	date={2013},
	number={2},
	pages={635--652},
	doi={10.1007/s00208-012-0859-0},
}

 


\bib{LV}{article}{
	author={Laface, Antonio},
	author={Velasco, Mauricio},
	title={A survey on Cox rings},
	journal={Geom. Dedicata},
	volume={139},
	date={2009},
	pages={269--287},
	doi={10.1007/s10711-008-9329-y},
}


\bib{PAGI}{book}{
	author={Lazarsfeld, Robert},
	title={Positivity in algebraic geometry. II},
	series={Ergebnisse der Mathematik und ihrer Grenzgebiete. 3. Folge. A
		Series of Modern Surveys in Mathematics [Results in Mathematics and
		Related Areas. 3rd Series. A Series of Modern Surveys in Mathematics]},
	volume={49},
	note={Positivity for vector bundles, and multiplier ideals},
	publisher={Springer-Verlag, Berlin},
	date={2004},
	pages={xviii+385},
	isbn={3-540-22534-X},
	doi={10.1007/978-3-642-18808-4},
}

\bib{PAGII}{book}{
	author={Lazarsfeld, Robert},
	title={Positivity in algebraic geometry. I},
	series={Ergebnisse der Mathematik und ihrer Grenzgebiete. 3. Folge. A
		Series of Modern Surveys in Mathematics [Results in Mathematics and
		Related Areas. 3rd Series. A Series of Modern Surveys in Mathematics]},
	volume={48},
	note={Classical setting: line bundles and linear series},
	publisher={Springer-Verlag, Berlin},
	date={2004},
	pages={xviii+387},
	isbn={3-540-22533-1},
	review={\MR{2095471}},
	doi={10.1007/978-3-642-18808-4},
}

 

 
 \bib{MU}{article}{
 	author={Mistretta, Ernesto C.},
 	author={Urbinati, Stefano},
 	title={Iitaka fibrations for vector bundles},
 	journal={Int. Math. Res. Not. IMRN},
 	date={2019},
 	number={7},
 	pages={2223--2240},
 	issn={1073-7928},
 	review={\MR{3938322}},
 	doi={10.1093/imrn/rnx239},
 }
 

 



\bib{PU}{article}{
	author={Postinghel, Elisa},
	author={Urbinati, Stefano},
	title={Newton-Okounkov bodies and toric degenerations of Mori dream spaces via tropical compactifications},
	year={2016},
	note={arXiv:1612.03861},
}


 
\bib{Th}{article}{
	author={Thaddeus, Michael},
	title={Geometric invariant theory and flips},
	journal={J. Amer. Math. Soc.},
	volume={9},
	date={1996},
	number={3},
	pages={691--723},
	issn={0894-0347},
}

\end{thebibliography}
\end{document}